\documentclass {article}
\usepackage[utf8]{inputenc} 
\usepackage[T1]{fontenc} 
\usepackage{amssymb,amsfonts,amsmath,amsthm,mathrsfs}

\usepackage{enumitem}
\usepackage{mathtools}
\mathtoolsset{showonlyrefs}  
\usepackage{eurosym}
\usepackage{amsfonts}
\usepackage{epsfig}
\usepackage{amssymb,amsfonts,amsmath,enumitem,mathrsfs,amsthm}
\usepackage{epstopdf} 
\usepackage{caption}
\newcommand\captionof[1]{\def\@captype{#1}\caption}
\usepackage{subcaption}
\usepackage{graphicx,fancybox,color}
\usepackage{wrapfig}
\usepackage{epstopdf}
\usepackage{tikz}
\usepackage{schemabloc}
\usepackage{pgf,tikz,tkz-tab,amsmath,amssymb,enumerate,float}
\usepackage[english]{babel}
\usepackage{amsmath}

\usepackage[nice]{nicefrac}
\usepackage{makeidx}
\usepackage{hyperref}
\usepackage[nosort, nocompress]{cite}
\makeindex

\def\eps{\varepsilon}

\newtheorem{theorem}{Theorem}[section]

\newtheorem{lemma}{Lemma}

\newtheorem{remark}{Remark}[section]

\begin{document}
\title{A phenotype-structured model for the tumour-immune response}
\author{Zineb Kaid $^{1,2}$, Camille Pouchol$^{3}$,
Jean Clairambault $^{4,5}$}

\maketitle
\begin{center}
$^{1}$ Laboratory of Biomathematics, University Djillali Liabes, BP 89, 22000 Sidi Bel Abbes, Algeria	\\
$^{2}$ Dynamical Systems and Applications Laboratory, Department of Mathematics, Faculty of Sciences, University Abou Bekr Belkaid, BP 119, 13000 Tlemcen, Algeria\\
$^{3}$ Université Paris Cité, MAP5 (UMR 8145), FP2M, CNRS FR 2036, F-75006 Paris, France\\	
$^{4}$  Laboratoire Jacques-Louis Lions, Sorbonne Universit\'e, UPMC Univ. Paris 06, CNRS UMR 7598,\\ 4, place Jussieu, BC 187, F75252 Paris Cedex 05, France\\
$^{5}$ Team Mamba, INRIA Paris, 2 rue Simone Iff, CS 42112, 75589 Paris cedex 12, France\\
\end{center}
\begin{abstract}
{This paper presents a mathematical model for tumour-immune response interactions in the perspective of immunotherapy by immune checkpoint inhibitors ({\it ICI}s). The model is of the nonlocal integro-differential Lotka-Volterra type, in which heterogeneity of the cell populations is taken into account by structuring variables that are continuous internal traits (aka {\it phenotypes}) present in each individual cell. These represent a lumped ``aggressiveness'', i.e., for tumour cells, malignancy understood as the ability to thrive in a viable state under attack by immune cells or drugs - which we propose to identify as a potential of de-differentiation -, and for immune cells, ability to kill tumour cells, in other words anti-tumour efficacy. We analyse the asymptotic behaviour of the model in the absence of treatment. By means of two theorems, we characterise the limits of the integro-differential system under an {\it a priori} convergence hypothesis. We illustrate our results with a few numerical simulations, which show that our model reproduces the three $E$s of immunoediting: elimination, equilibrium, and escape. Finally, we exemplify the possible impact of {\it ICI}s on these three {\it E}s.}
\end{abstract}
\noindent {\textbf{Keywords:}}{ Tumour-Immune interactions, Phenotype-structured model, Asymptotic analysis, Immune checkpoint inhibitors.}\hspace{.2cm}
2020 MSC: 35B40, 35F50, 35Q92, 92-10, 92C50, 92D25
\section{Introduction}
In the field of oncology, several clinical and experimental studies concur to show that the immune system plays a decisive role, providing tumour control, long-term clinical benefits and prolonged survival~\cite{Stewart}. Nevertheless, the anti-tumour immune response is an extremely complex process that depends on many factors. In this context, mathematical models can help understand the interactions between tumour growth and the immune response. 

In the present article, we propose an integro-differential equations based model, designed to analyse tumour-immune interactions between cell populations and the asymptotic behaviours of these populations. We follow the principle of modelling cell population heterogeneity by structuring them by relevant internal traits (aka cell {\it phenotypes}), as initiated in \cite{Lorz1} and partially reviewed in \cite{JCCPBioMath}. 

Deterministic phenotype-structured models, which are usually stated in terms of non-local partial differential equations or integro-differential ones have been widely used to describe phenotype heterogeneity in tumour cell populations. In these models, the phenotypic state is represented by a continuous real variable, modelling different biological characteristics such as viability and fecundity, see~\cite{Almeida_et_al, Alvarez_et_al, RCLCJ, P-al} and the references therein. 

To the best of our knowledge, the first phenotype-structured model for tumour-immune interactions was proposed by Delitala and Lorenzi~\cite{LTDM}, where a tumour cell population characterised by heterogeneous antigenic expressions is exposed to the action of antigen-presenting cells and immune T-cells. More precisely, their model incorporates five populations of cells. All populations but one are assumed in~\cite{LTDM} to be structured by a real continuous variable in $[0,1]$ representing an internal phenotypic state of the cell. Their model reproduces well the selective recognition and learning processes in which immune cells are involved. Our model, which belongs to the category of non-local Lotka-Volterra systems, has been designed to offer a rationale for the use of immune checkpoint inhibitors~\cite{Schreiber-al}. 

Such models, which can be derived from stochastic individual-based models~\cite{Champagnat}, are known to possibly lead to the concentration of populations on one or several phenotypes, which will be shown in the model described below for the tumour cell population, under restrictive assumptions. The original motivation for this work is the article~\cite{P-al}, in  which an integro-differential system for the time evolution of densities of cancer and healthy cells, structured by a continuous phenotypic variable, representing their level of resistance to chemotherapy to which they are exposed, is studied. In a completely different context, which is the application to $ICI$ immunotherapy, we will here make use of similar methods of asymptotic analysis. 

\paragraph{Main theoretical and numerical results.} In summary, we identify the possible (generally unique) non-trivial limits the solution to the integro-differential system may have. More precisely, under the strong {\it a priori} assumption that the density of cancer cells converges, we prove that its limit, when non-zero, is a weighted Dirac mass. We provide in Section~\ref{sec4} a formula to compute the weight and location.

Moreover, we present simulation results that show how our model illustrates the three $E$s of immunoediting (elimination, equilibrium, and escape) and that it may also exhibit oscillatory solutions. We mention that, in the context of our model, which is of the non local Lotka-Volterra type, it may be difficult to distinguish between equilibrium and escape. Nevertheless, we propose to interpret solutions for which the system reaches its carrying capacity as tumour escape, whereas solutions for which the tumour cell population is contained below its carrying capacity may be interpreted as an equilibrium between tumour and immune cells.

\paragraph{Outline of the paper.} The paper is organised as follows. In Section \ref{sec2}, we start by introducing biological motivations for the development of the model under study. The integro-differential model itself is presented in detail in Section \ref{sec3}. We then analyse the model and prove some asymptotic properties in Section \ref{sec4}. In Section \ref{sec5}, we present some numerical results. In Section \ref{sec6}, we conclude with several comments and open questions.

\section{Biological background}\label{sec2}

When a cancer cell population thrives, the immune response, and essentially its part that is constituted of CD8+ T-lymphocytes (for the adaptive response) and NK-lymphocytes (for the innate response), consists in recognising as foe elements and killing these cancer cells. This has been called immunosurveillance, later immunoediting~\cite{Bertolasobook, Pradeubook, Schreiber-al}, which may consist of three different configurations: eradication, equilibrium or escape. If this process is performed during the early stages of tumour initiation, the tumour is quickly and successfully eradicated. The immune response may also only contain the development of tumour, stabilising it at an equilibrium mass between eradication and maximum tumour carrying capacity, without eliminating it. However, cancer cells can escape these innate NK-cell and adaptive specific T-cell immune responses in the course of genetic and phenotypic evolution at the time scale of a cancer disease, and lead the total tumour mass to its maximum carrying capacity. More precisely, phenotypic\footnote{The term phenotype means here the set of characteristics of an individual (a cell, represented by a population density function taking a given value, in our model), resulting from the interaction of its genotype with the environment; `observable' characteristics are not necessarily visibly observed, but may be internal traits, e.g., related to some epigenetic, reversible, modification by graft of a chemical radical (methyl, acetyl...) on some base of the DNA before transcription, or on some amino acid in a histone protein, such traits being possibly evidenced after some dynamic stimulation only.} heterogeneity in the cancer cell population, involving its possible internal invasion by secondarily mutated, robust, cells, may be responsible for both tumour escape and treatment failure. Here, a focus is set on immunotolerance~\cite{Pradeubook, Schreiber-al}, which renders cancer cells able to evade immune detection and elimination. Indeed, cancer cells have the resource to weaken the immune response by emitting molecules such as PD-L1 (i.e., PD-1 ligand) and CTLA-4 \footnote{PD-1: programmed cell death protein 1, a receptor located on NK- and T-lymphocytes; CTLA-4: cytotoxic T lymphocyte antigen 4 for the adaptive response of T-lymphocytes.} which can respectively bind to the PD-1 and B7 receptors on activated T-cells, inhibiting their cytotoxic activity and reducing their own immunogenicity. This is represented in the model by direct competition between cancer cells and T-lymphocytes. As regards innate immunity, the role of NK-lymphocytes (readily effective on cancer cells with lacking MHC-I surface antigens~\cite{Guzman2020}, not the same mechanism as in the case of CD8+ T-cells, that are activated by tumour antigen-presenting cells, APCs), has recently gained more consideration, in particular, because a role for anti-PD1/anti-PDL1 has been suspected for them in cancer immunotherapies~\cite{Pesce2019}.

Immunotherapy with immune checkpoint inhibitors~\cite{Schreiber-al} (hereafter noted {\it ICI}s) is a recently introduced class of drugs (aiming at being less toxic than the classical anti-cancer therapies, chemotherapy and radiation therapy) that inhibit such cancer cell-produced inactivation of T- and NK-lymphocytes, either at the receptor sites on lymphocytes or by inhibiting the ligands themselves. The clinical use, firstly of anti-CTLA4 drug Ipilimumab, which has been shown to mainly target the priming lymphocyte response at the level of lymphoid organs~\cite{Zhang2021}, and later of direct (at the tumour site) antagonisers of PD-L1 to PD-1 binding, drugs Nivolumab and Pembrolizumab~\cite{Han2020}~\footnote{Note that in the sequel, as our model aims at representing direct tumour-immune interactions and their possible enhancing by immunotherapy, the term {\it ICI}s represents this second class of drugs, anti-PD1 anti-PDL1 drugs.}, has drastically modified the prognosis of several advanced cancers that were until recently out of reach (e.g., melanoma, a skin cancer with very bad prognosis~\cite{Robert-al}), offering sustained positive responses (about $20\%$ of complete cures, the remaining 80\% consisting of non- or partial responders with relapse).  However, not all cancer types respond as well as melanoma. To the best of our knowledge, the reasons for successes or failures are still unknown. Moreover, there is no clear dose-response relationship and a maximum tolerated dose, for checkpoint inhibitors, has not been identified as yet~\cite{Le_Louedec_et_al}. 
 
\section{The model}{\label{sec3}}
\subsection{Continuous phenotype-structured model for tumour-immune interactions}

The main interest of structuring a cell population dynamic model with continuous phenotypes is to take into account the continuously evolving phenotypes of interest in each individual cell, responsible for their heterogeneity and plasticity~\cite{JC4ISMCO, JCECC23} in the cell population, with respect to relevant functional traits determining their adaptation to changing conditions in the tumour micro-environment (such as externally control delivery of drugs) and to their mutual interactions. In our case, we consider as relevant to represent in the NK- and T-lymphocyte population a phenotype $y$ coding continuously for their anticancer predating power, that can be considerably weakened by PD-L1 ligands emitted by tumour cells, and, as regards tumour cells, a lumped phenotype $x$ coding for their malignancy, which may be evaluated by their ability to de-differentiate. Such ability for individual cancer cells to de-differentiate, i.e., to gain a status closer to so-called stem cells by swimming upstream in their normally unique-sense maturation stream within their cell lineage is well known in cancer~\cite{Shen}. It allows cancer cells to gain plasticity and thus, by reactivating defence mechanisms that are normally silenced in healthy cells, to escape the predation of NK- and T-lymphocytes at a tumour site.

To grab more precisely the idea of de-differentiation, one must have in mind that normal cells differentiate (i.e., mature) within their lineage from pluripotent, immature, stem cells to terminally differentiated, fully functional, cells, and that such differentiation is physiologically irreversible (however not any more in cancer cells, in which mutations are not at all necessary for such reversal mechanism, which relies on fast hijacking of epigenetic enzymes~\cite{Sharma} or transcription factors~\cite{Leeetal2020}). The best studied cases (documented in any textbook of cell biology, see, e.g.,~\cite{Zon2001} for haematopoiesis and~\cite{WA1984} for epithelial cells) of such physiological maturation are, in the bone marrow, the haemotopoietic lineages from pluripotent haematopoietic stem cells to, e.g., mature granulocytes or lymphocytes, and, in the gastrointestinal tract, the enterocyte lineage from crypt cells to mature enterocytes of villi in the jejunum. The ability of cancer cells to swim up this normally one-way maturation stream is one of the main phenomena referred to when mentioning their plasticity. Another well-known case of plasticity in cancer is transdifferentiation (i.e., jumping directly from one cell phenotype to another one without following any upstream way), in particular epithelial to mesenchymal transition (EMT, and its reverse MET), by which epithelial cells acquire a mesenchymal phenotype, moving in a fibroblast-like way far from their normal layers, a fundamental process for making cancer metastases~\cite{Kalluri2009}. In phenotype-structured cell population dynamic models, as in~\cite{Alvarez_et_al}, a continuous plasticity phenotype will more generally quantify the ability of a cell to change its phenotype in an adaptive response to changes in its environment.

Such representation by phenotype-structured cell population models of the adaptation of individual cells to a changing environment is not new, and it has already been studied in depth, from the mathematical point of view, in a book by B. Perthame~\cite{Perthame} and illustrated in a variety of articles in the last decade~\cite{Alvarez_et_al, AlvarezJC2023, Ardashevaetal2020,  Chisholm2015, Chisholm2016a, Chisholm2016b, JCCPBioMath, LTDM, RCLCJ, Lorz1, Pouchol-Lorenzi, P-al}. Of note, the consideration of heterogeneity and plasticity in cancer cell populations has lately become of high interest to cancer biologists~\cite{Meacham2013, Tang2012, Yabo2022}, and we suggest that a proper way to take them into account in mathematical models is by studying such adaptive phenotype-structured cell population dynamic equations, as proposed, e.g., in~\cite{JCECC23, Shen}. We do not contend that such modelling is universally popular in the community of cancer biologists, of course, but that at least some of them consider these models with interest and feed mathematicians with the physiological and physiopathological knowledge they need to design biologically relevant mathematical models (that the converse feedback from mathematics towards experimental biology and clinical therapeutics is still in its infancy is another story).

Clearly, direct biological measurements of the expression of such traits as the ones mentioned above are usually lacking (and the same is true about the expression of traits of drug resistance), so that they may be assessed, in the best cases, only by indirect concentrations of proteins~\cite{Leeetal2020} (or of their activity when they are enzymes or transcription factors), or expression of genes~\cite{vanderLeun2020}, indeed not easy to measure experimentally, or else they must be taken as hidden variables of functional nature, such as viability, fecundity, plasticity, motility, that determine cell population fate. Here, we choose to structure our equations representing the dynamics of cell populations under study by the hidden continuous functional variables $x$ (malignancy in tumour cells) and $y$ (anti-tumour aggressiveness, i.e., efficacy, in immune cells) presented above. 

\subsection{Biological motivations}
In the model relying on the system of three phenotype-structured equations shown in Subsection~\ref{subsec_model}:
\begin{itemize}
\item For the tumour cell population of density $n(t,x)$, a cell of phenotype $x$ is all the more malignant, i.e., able to thrive, as $x$ is close to $1$, and conversely less malignant when $x$ is close to $0$. More precisely, the malignancy trait $x$ represents a progression potential towards stemness (ability to de-differentiate, aka plasticity).

Let us mention that the malignancy trait $x$ might in principle be measured in single cells
by assessing the expression of genes like the Yamanaka genes, identified in 2006, that enable de-differentiation, yielding induced pluripotent stem cells from embryonic or even mature mouse fibroblasts~\cite{YamanakaCell2006}. More recently, a de-differentiated phenotype $MIT^{low}/AXL^{high}$ phenotype, defined by the concomitant downregulation of the transcription factor $MIT$ and accumulation of the tyrosine kinase receptor $AXL$, has been evidenced in immunotherapy-resistant melanoma cells~\cite{Leeetal2020}, which could provide a measurable basis for such continuous malignancy trait $x$ identified as a potential for tumour cells to de-differentiate in response to deadly attacks coming from the immune response or more generally from the tumour microenvironment, including drugs. Importantly, we assume in this model that both the density of a loss-of-self in tumour cells sensed by NK-cells 
 (made precise in the next paragraph) and the density of specific tumour antigens sensed by APCs reflect the level of the {\em hidden} tumour aggressiveness, or malignancy, phenotype $x$ in the tumour cell population, even though the anti-tumour action of lymphocytes will be directed towards the {\em manifest} general loss-of-self (for NK-cells) or specific tumour antigen-bearing cells (for T-cells).
\item For the T-lymphocyte population and for the non-adaptive NK-lymphocyte population, $\ell(t,y)$ for effector cells and $p(t,y)$ for na\"ive cells in the system of equations~\eqref{ID} shown in Section~\ref{subsec_model} below, in a similar way, we structure it by a phenotype $y$ of anti-tumour aggressiveness, or efficacy, which may be defined as the reverse of the `dysfunction' or `exhaustion' phenotype that has been observed in CD8$^+$ T-cells exhibiting incapacity to efficaciously fight tumour cells. In our model, the difference between NK-lymphocytes and effector (CD8+) T-cells consists in the nature of their action on tumour cells, either independent of the tumour phenotype $x$ for NK-cells represented below by a function $\varphi(t)$, or highly dependent on it for T-cells, represented by a function $\varphi(t,x)$. In the analysis of our model, we will study separately the case of innate (function $\varphi(t)$) and adaptive (function $\varphi(t,x)$) immune response, and also a mix of these two cases. To identify and measure in single cells such dysfunction or exhaustion in T-cells, different biological markers have been proposed; they have been recently reviewed in~\cite{vanderLeun2020} (an article in which it is, in particular, noted that {\it ``T cell dysfunctionality is a gradual, not a binary, state''}, which fully justifies the continuous character of our structure variable $y$). The closer the phenotype $y$ approaches $0$, the less aggressive are T- and NK-lymphocytes, i.e., less competent to kill cancer cells (complete exhaustion), whereas if $y$ approaches $1$, they are highly aggressive (full competence) against the targeted tumour cells, an aggressiveness identified by their competence as immune cells due to the tumour antigen recognition performed by the APCs and transmitted to na\"ive T-cells in lymphoid organs, or to the absence of MHC-I\footnote{The Major Histocompatility Complex I, MHC-I, is present in all jawed vertebrates, hence in Man, species to which we will limit the scope of our model, which is intended to pave the way for cancer immunotherapy.} antigens (loss-of-self) in tumour cells in the case of NK-cells. The principle of immune checkpoint inhibitor ({\it ICI}) immunotherapy is to boost CD8+ T-lymphocytes and NK-lymphocytes in their efficacy by antagonising such tumour-emitted inhibitory mechanisms, mainly, in the modelling framework presented here, PD-L1 to PD-1 binding on T-cells and on NK-cells.
\end{itemize}

\subsection{Modelling choices for the mathematical functions of phenotypes $x$ and $y$}
In the absence of experimental data, the precise choices for functions $r, d, \mu, \nu, \varphi$, used in the system~\eqref{ID} shown in Subsection~\ref{subsec_model} below, are largely arbitrary, only guided by physiological considerations on an assumed monotonicity.
They are listed in Table~\ref{tab1} of Section~\ref{sec5} for simulations, reflecting such monotonicity: non-increasing for $r, d, \mu, \nu$, non-decreasing for the weight function $\psi$ that defines the immune response $\varphi$ in the case of innate immunity by NK-cells (see below). The biological background for these functions is as follows. 
\begin{itemize}
\item We assume that in the absence of immune response, tumour cells undergo logistic growth, with a net growth rate (aka fitness) defined by
\begin{equation}
\label{fitness}
r(x)-d(x)\rho(t).
\end{equation}
Here, the function $r(x)$ stands for the intrinsic proliferation rate. As $x$ stands for a de-differentiation, stem-like, cell phenotype, admitting that a stem-like status does not favour replication velocity, $r$ will typically be assumed to be a positive, decreasing function of $x$ on $[0,1]$, e.g., of the form $r(x)=r_0-\eta x^2$
where the parameter $r_0>0$ corresponds to the maximum fitness of cancer cells, while $\eta>0$ provides a measure of the strength of natural selection in the absence of the immune response, with $r_0- \eta > 0$. The term $d(x)\rho(t)$ models the intrinsic death rate due to within-population competition for space and resources, assumed to be proportional to the total population mass of tumour cells $\rho(t)$. The function $d$  will typically be taken to be a positive, decreasing function of $x$ on $[0,1]$ (in the same way as for the replication function $r$, a de-differentiated, stem-like, status is admitted to protect cells from the natural death term represented by the function $d$).

The fitness structure chosen here for the tumour and for the immune cell population is of the nonlocal Lotka-Volterra type. It has been in particularly used in~\cite{Perthame} to model the adaptation of individuals to their environment. 
\item We assume that, once an immune response has been activated, the tumour cells interact with NK-cells as an added death term at a rate which is proportional to the product of the tumour cell population density by a weighted integral $\varphi(t)$ given by
\begin{equation}
\label{non-adaptive}
\varphi(t)=\displaystyle\int_{0}^{1}\psi(y)\ell(t,y)\, dy,
\end{equation} 
where $\psi$ is a positive function, which we will take to be non-decreasing on $[0,1]$, $\mu(x)$ being a sensitivity function. We note that in this formulation, the immune response $\varphi(t)$, emitted by NK-cells present at the tumour site, is non-specific, only the tumoral sensitivity function $\mu(x)$ makes it somehow specific; the function $\varphi(t)$ then stands for a response in which the phenotype $y$ in lymphocytes is averaged over all the  population of activated NK-cells, that are sensitive to to the {\it loss-of-self} in tumour cells, $\varphi(t)$ representing a sort of ``mass immune response''. 
\item In order to account for an adaptive, specific, immune response which is the one of T-cells, we more generally consider the immune reaction function $\varphi$ to be of the form
\begin{equation}
\label{eq-varphi}
\varphi(t,x)=\int_{0}^{1}\Psi(x,y)\ell(t,y)\, dy.
\end{equation}
Here, the weight function $\psi(y)$ of the innate response is replaced by $\Psi(x,y)$, which we typically take to be the product of a function $\psi(y)$ by a localisation kernel, e.g., 
\begin{equation}
\label{adaptative-function}
\Psi(x,y)=\frac{\psi(y)}{v}e^{-|x-y|/v},
\end{equation}
in which $\psi$ will again be a positive function, non-decreasing on $[0,1]$, assumed for simplicity to be the same as in the innate, non-adaptive case~\eqref{non-adaptive}. Parameter $v$ is the precision with which the immune response targets the cancer cell population (as identified by its malignancy trait~$x$, representing specific tumour antigens borne by the cancer cells).

We will in fact in simulations consider separately these two cases, native non-specific (NK-cells: $\varphi(t)$ given by~\eqref{non-adaptive}) and adaptive specific (T-cells: $\varphi(t,x)$ given by~\eqref{eq-varphi}) anti-tumour immune response, and also a mixed case, convex combination of the two immune responses, non-specific (NK-cells) and specific (T-cells):
\begin{equation}
\label{combary}
\Psi_{\lambda}(x,y)=\left((1-\lambda)+\lambda \frac{1}{v}e^{-|x-y|/v}\right)\psi(y), \quad \lambda\in [0,1],
\end{equation}
corresponding to simultaneous and independent activation of NK-cells and T-cells by loss-of-self (NK-cells) and specific tumour antigen (T-cells) stimuli. This choice interpolates (function~$\psi$ being fixed) between~\eqref{non-adaptive} obtained with $\lambda = 0$ (for NK-cells), and~\eqref{eq-varphi}, with the choice~\eqref{adaptative-function}, obtained with $\lambda = 1$ (for T-cells).
\item The function $\mu(x)$ mentioned above represents a  factor of sensitivity to the effects of the immune response. As de-differentiation is supposed to protect tumour cells from these effects (e.g., by hiding tumoral antigens, targets of lymphocytes), $\mu$ will be a positive decreasing function of $x$.
\item The amplification of the na\"ive T-lymphocytes $p(t,y)$ at lymphoid organs is related to the mean $x$ malignancy value through a weighted integral $\chi(t,y)$ of the tumour cell population, representing  the message borne by APCs to initiate the adaptive anti-tumour immune response produced in the lymphoid organs. When an APC detects a tumour cell, the related antigen is presented to na\"ive T-cells in lymphoid organs. Thus, na\"ive T-cells that recognise this antigen as their cognate one become activated and start their ammplification, ie., they start to proliferate and, through a complex process chain, they become able to recognise and attack tumour cells that express the cognate antigen. The function $\chi(t,y)$ is defined as
\begin{equation}
\label{eq-chi}
\chi(t,y)=\displaystyle\int_{0}^{1}\omega(x,y)n(t,x)\,dx,
\end{equation}
where $\omega$ is another localisation kernel such as $\omega(x,y) = \alpha \frac{1}{s}e^{-|x-y|/s}$,
so as to represent a more or less faithful tumour antigen detection message transmitted  by the APCs (whose mission is to activate na\"ive T-lymphocytes) to the lymphoid organs about both the size and malignancy of the tumour. The efficacy of activated T-lymphocytes in killing tumour cells depends on the initial size of the tumour and on how localised the kernel is (i.e., on the width of the range of phenotypes $y$ concerned by their detected tumour cognates~$x$, which can be measured by the value of the parameter $s$ in the proposed function $\omega(x,y)$). The parameter $\alpha$ represents the strength of the immune response. In the present model, communication between recognition at the contact of tumour cells and activation of na\"ive T-lymphocytes at the site of lymphoid organs will be represented, for the sake of simplicity without considering any delay, by the shortcut of the function $\omega(x,y)$. In this localisation kernel function, the parameter $s$ may be seen as the precision (all the higher as $s$ is lower) of the detection of the malignancy trait $x$ in the cancer cell population by APCs or circulating NK-cells.
\item We consider a similar mechanism for NK-lymphocytes, that are known to proliferate and amplify not only in the bone marrow but also in lymphoid organs, in the same way as T-lymphocytes do~\cite{Abel2018}. In the case of this innate immune response, there are no APCs, but the message from sensor patrolling NK-lymphocytes to  proliferating NK-lymphocytes in lymphoid organs is assumed to be of (coarse, quantitative) humoral nature, carrying a message on the density of loss-of-self loci in the tumour cell population. The same function $\chi(t,y)$ with the same localisation kernel $\omega(x,y)$ will be used for the activation of NK-lymphocytes.
\item In the second equation of \eqref{ID} for the competent NK- and T-lymphocytes $\ell(t,y)$, the sensitivity function $\nu(y)$ of the anti-tumour aggressiveness phenotype $y$ represents the weakening of both categories of lymphocytes (immunotolerance) induced by PD-1 ligands, this sensitivity function multiplying the total mass of tumour cells $\rho(t)$; note that it is also assumed to be decreased by {\it ICI}s present at the denominator. As the function $\nu$ stands for a sensitivity factor in lymphocytes to the weakening reaction molecules (in this model, mainly PD-L1) emitted by tumour cells or produced in the tumour microenvironment, it will be chosen to be a positive, decreasing function of $y$, which in this case reflects the fact that cells in the phenotypic state $y = 1$ are fully aggressive on contact with tumour cells and, for cells in phenotypic states other than the most aggressive one, the inhibition term induced by the tumour cells decreases with the drug dose.
\item The population of NK-lymphocytes and na\"ive T-lymphocytes residing in the lymphoid organs is supposed to be regulated in a logistic way by a logistic term $k_1p(t,y)$, where the parameter $k_1$ stands for the natural death rate of the population of lymphocytes imposed by carrying capacity constraints (e.g., limited availability of space and resources in lymphoid organs).
\item The input of external control targeting immune checkpoints inhibitors is represented by the function $ICI(t)$ that enhances anti-tumour CD8+ T-lymphocyte and NK-lymphocyte responses by boosting the exhausted immune cells, which helps them to respond strongly to the presence of the tumour, by ``weakening the weakening'' immunotolerance induced by the tumour cells. We assume that 
\begin{equation}
0\leq ICI(t)\leq ICI^{max}. 
\end{equation}
for some maximum tolerated dose $ ICI^{max}$.
The factor $\textstyle \frac{1}{1+h ICI(t)}$, with $h>0$, tunes the decrease in the immunotolerance  rate, decrease due to the immune checkpoints inhibitors therapy. We note that fine details of clinical administration protocols are not meant to be described here. We also mention that {\it ICI} is a quantitative dose function, and is APC-independent.  
 
\end{itemize}
\subsection{An adaptive cell population dynamic model}
\label{subsec_model}
To describe tumour-immune interactions, we consider three different cell population densities: 
\begin{itemize}
\item a heterogeneous cancer cell population $n(t,x)$ with continuous aggressiveness (or malignancy) trait $x\in[0,1]$ linked to their stemness (i.e., their ability to de-differentiate, allowing them to re-differentiate with adapted phenotypes); they are endowed with a natural (logistic and nonlocal) death term $d(x)\rho(t)$ representing within cell population competition for space and d nutrients, and an added death term due to anti-tumour immune cell predation $\mu(x)\varphi(t,x)$, where $\mu(x)$ $\mu(x)$ and $\varphi(t,x)$ are the sensitivity function and the immune predation function described above, with the target localisation kernel for T-lymphocytes included in the function $\Psi_{\lambda}$  for $\lambda\neq0$ .
\item a heterogeneous population of mixed competent T-lymphocytes and NK-lymphocytes $\ell(t,y)$ endowed with continuous anti-tumour aggressiveness trait $y$ ranging from $0$ (exhausted) to $1$ (highly aggressive) interacting with cancer cells $n(t,x)$ at the tumour site. Here, function $\nu(y)$ tunes the immunoediting function from tumour cells, that weakens the aggressivity of NK- and T-lymphocytes by PD-L1 on PD1 receptors on both lymphocyte populations.  
\item a heterogeneous population of na\"ive T-lymphocytes and inactive NK-lymphocytes $p(t,y)$, either resident and present at the tumour site (for NK-cells, particularly activated by their sensing lack of MHC-I surface antigens in tumour cells, so-called ``loss-of-self''), or present in distant lymphoid organs, informed there of the presence of tumour cells of malignancy phenotype $x$ by patrolling NK-lymphocytes - or by humoral messages - for inactive NK-lymphocytes, and for na\"ive T-lymphocytes by \textbf{APCs} (antigen-presenting cells, here represented by a weighted integral of the cancer cell population involving the localisation kernel $\omega(x,y$). Both ``na\"ive'' cell populations are represented by the lumped (i.e., gathering NK- and T-lymphocytes in a population of still inactive immune cells) population density $p(t,y)$. In our asymptotic analysis and in simulations, we consider separately the three cases: innate, adaptive, and a combination of the two immune responses. 
\end{itemize}
Our model is given by the following system of integro-differential equations (IDEs):
\begin{equation}
\label{ID}
\left\{
\begin{array}{ll}
\displaystyle\frac{\partial n}{\partial t}(t,x)=\left[r(x) -d(x) \rho(t)-\mu(x)\varphi(t,x)\right]n(t,x), \quad \hbox{for} \quad t>0, x\in [0,1], \\[0.3cm]
\displaystyle\frac{\partial \ell}{\partial t}(t,y)=p(t,y)-\left(\frac{\nu(y)\rho(t)}{1+hICI(t)}+k_1\right)\ell(t,y), \quad \hbox{for} \quad t>0, y\in [0,1],\\[0.3cm]
\displaystyle\frac{\partial p}{\partial t}(t,y)=\chi(t,y)p(t,y)-k_2p^{2}(t,y),\quad \hbox{for} \quad t>0, y\in [0,1].
\end{array}%
\right.
\end{equation}
with total mass of cancer cells at time $t$
\begin{equation}
\rho(t):=\displaystyle\int_{0}^{1}n(t,x)\,dx.
\end{equation}
The initial value function $n(0,x)$ is chosen to represent the assumed initial malignancy of the tumour, and in the same way, the initial value functions $\ell(0,y)$ and $p(0,y)$ will be chosen to represent the initial host's immune response.\\
\mbox{}\\
Having this model in mind, our goals in the present study are 
\begin{itemize}
\item to analyse the asymptotic properties of the model, as we want to understand how the interaction between tumour cells and T cells leads to the selection (or not) of some traits, which are considered as dominant traits by the environment; 
\item to numerically investigate if and how our model captures the three $E$s of immunoediting, i.e., eradication, equilibrium and escape.
\end{itemize}

\subsection{Comparison with an ODE-reduced system}
In order to exploit useful ideas to guide our study of the dynamics of the above integro-differential system, we mention that a simplified version of \eqref{ide}, reduced to an ODE system, has been analysed in~\cite{Kaidetal}. Assuming that all functions are constant in $x$ and $y$, and denoting\begin{equation}
\sigma(t):=\int_{0}^{1}\ell(t,y)\,dy, \quad \hbox{and}\quad \gamma(t):=\int_{0}^{1}p(t,y)\,dy,
\end{equation} the system \eqref{ID} boils down to the dynamics of $t \mapsto (\rho(t),\sigma(t),\gamma(t))$, which after integration solves the following ODE system: 
\begin{equation}
\label{odesys}
\left\{
\begin{array}{ll}
\displaystyle\frac{d\rho(t)}{dt}=\left[r-d\rho(t)- \mu \sigma(t)\right]\rho(t), & \\[0.3cm]
\displaystyle\frac{d\sigma(t)}{dt}=\gamma(t)-\left(k_1+\nu\rho(t)\right)\sigma(t),& \\[0.3cm]
\displaystyle\frac{d\gamma(t)}{dt}=\gamma(t)\left(\omega \rho(t)-k_2\gamma(t)\right).
\end{array}%
\right.
\end{equation} 
The mathematical analysis of these equations has been performed in~\cite{Kaidetal}, in the particular case where $k_1=\nu$, $\mu = \omega = 1$. The existence  of the steady states has been characterised and analysed with respect to their local asymptotic stability. Regions of the parameter space have also been identified, in which a Hopf bifurcation exists. The ODE system~\eqref{odesys} reproduces a tumour equilibrium (second situation of the immunoediting process), which corresponds either to a stable steady state or to a stable limit cycle, characterised by a sustained periodic behaviour of alternating growth and decay (without extinction) of both tumour and immune T cells. For particular choices of initial conditions, the ODE model also captures either tumour eradication or tumour immune escape.
\begin{figure}[H]
	\centering{
		\includegraphics[height=6cm]{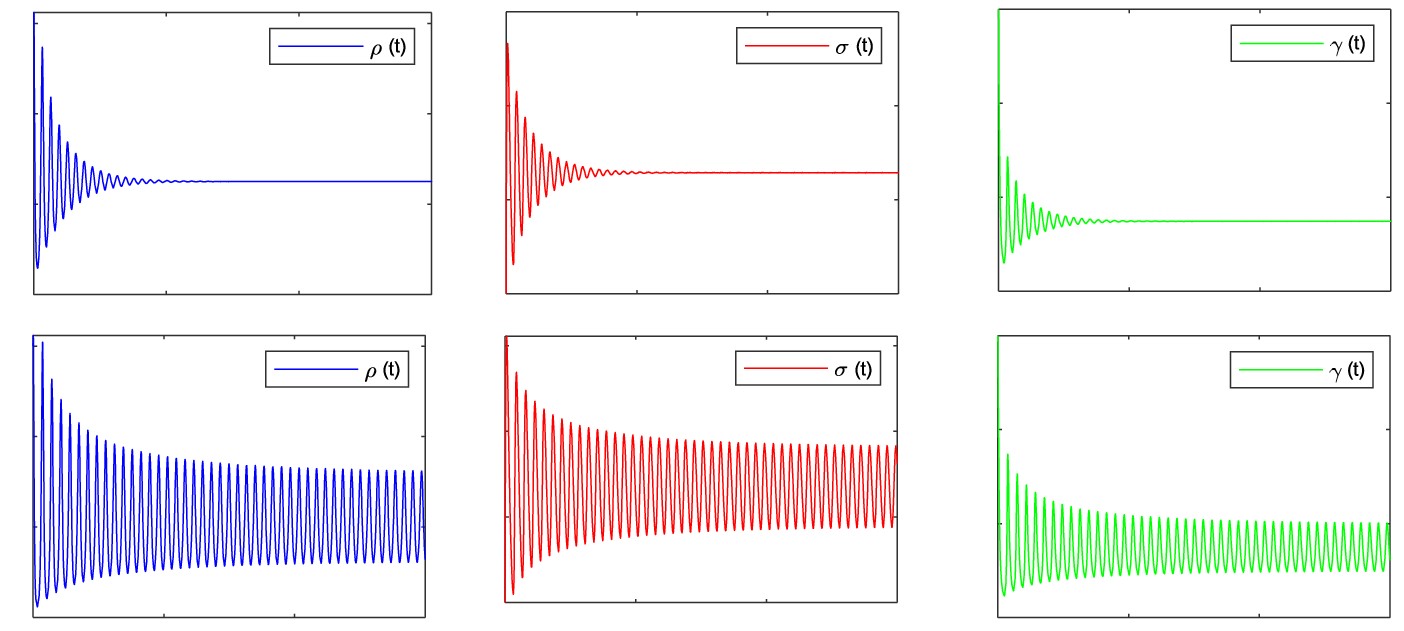}
		\caption{\label{odefig}{\small{{\bf ODE reduction} (adapted from~\cite{Kaidetal}): Plots displaying the time evolution of total masses $\rho(t)$ (left panel), activated NK-cells and competent T-cells $\sigma(t)$ (central panel), and na\"ive T-cells or inactive NLK-cells $ \gamma(t)$ (right panel) as defined by system (\ref{odesys}) in two different cases: stationary and periodic solutions~\cite{Kaidetal}}. {\bf{Upper row}}. The solution shows stability of the interior equilibrium $(0.6257,1.1436,0.746)$, for $k = 0.8514$. {\bf{Lower row.}} The solution shows instability of the interior equilibrium $(0.5204,1.1699,0.7115)$ with limit cycle (Hopf bifurcation), for $k=0.7314$. For all plots, $r=1.3, d=0.25$, $\nu=0.4$, and initial conditions are $(\rho_0,\sigma_0, \gamma_0) = (1.5,0.5,3) $.}}}
\end{figure}

\section{Asymptotic analysis}\label{sec4}

\subsection{Asymptotics in the absence of treatment: innate, non-adaptive response}
We study the asymptotic properties of the system (\ref{ID}) in the absence of treatment, i.e., with $ICI(t) = 0$. Of course, upon changing the function $\nu$, our study also encompasses the case where the dose $ICI$ is taken to be constant with time.\\
The evolution of the population densities is then governed by the following integro-differential system:
\begin{equation}
\label{ide}
\left\{
\begin{array}{ll}
\frac{\partial n}{\partial t}(t,x)=\left[r(x)-d(x)\rho(t) -\mu(x)\varphi(t)\right]n(t,x), \quad \hbox{for} \quad t>0, \; x\in [0,1],& \\[0.3cm]
\frac{\partial \ell}{\partial t}(t,y)=p(t,y) -\left(\nu(y)\rho(t)+k_1\right)\ell(t,y), \quad \hbox{for} \quad t>0, \; y\in [0,1] ,&\\[0.3cm]
\frac{\partial p}{\partial t}(t,y)=\chi(t,y)p(t,y)-k_2p^2(t,y), \quad \hbox{for} \quad t>0, \; y\in [0,1],\end{array}%
\right.
\end{equation}
the above system starting from initial conditions
\begin{equation}
\label{initialeconditions}
n(0,x)=n^0(x)\geq 0, \quad \ell(0,y)=\ell^0(y)\geq 0, \quad p(0,y)= p^0(y)\geq 0.
\end{equation}
\paragraph{Main assumptions on the functions and initial conditions.}
For the remaining part of this section, we assume that the initial conditions $n^{0}$, $\ell^{0}$ and $p^{0}$ are all in $\mathcal{C}([0,1])$, and whenever  necessary, we will assume that
\begin{equation}
\label{H4}
n^{0}>0  \quad \hbox{and} \quad p^{0}>0 \quad \hbox{on} \; [0,1],
\end{equation}
and we will work with the following regularity assumptions:
\begin{equation}
\label{H3}
r,d,\mu,\psi,\nu \in \mathcal{C}([0,1]), \quad  \hbox{and} \quad \omega \in \mathcal{C}([0,1]^2),
\end{equation}
and all the above functions are assumed to be positive. In the more general adaptive case, the assumption $\psi \in \mathcal{C}([0,1])$ is replaced by $\Psi \in \mathcal{C}([0,1]^2)$. We note that all proposed functions in the introduction and those used in simulations do satisfy these regularity and positivity hypotheses.

We also stress that no monotonicity assumptions whatsoever are required for the results of this section to hold true.

The existence and uniqueness of global classical (nonnegative) solutions in $C^0([0,+\infty), L^1(0,1)^3)$ is standard and follows from using the Banach fixed point theorem, see \cite{Perthame}. Given our regularity hypotheses for initial conditions, it is then clear that functions $n(t, \cdot)$, $\ell(t,\cdot)$ and $p(t,\cdot)$ are all continuous on $[0,1]$, at all times $t \geq 0$.


\subsubsection{Asymptotics for tumour cells alone}
In the absence of immune response (for instance, assuming either that there are no competent immune cells initially, i.e., $\ell^0 = 0$, or that immune cells are inefficient in interacting with cancer cells through either $\psi = 0$ or $\mu=0$), the first equation of~\eqref{ide} boils down to a standard logistic integro-differential model, namely
\begin{equation}
\label{singleide}
\left\{
\begin{array}{ll}
\frac{\partial n}{\partial t}(t,x)=\left[r(x)-d(x)\rho(t)\right]n(t,x), \quad 
n(t=0,x)=n^{0}(x)\geq 0,&\\[0.3cm]
\rho(t)=\int_{0}^{1}n(t,x)\,dx. 
\end{array}%
\right.
\end{equation}
The asymptotic behaviour of this equation is well known~\cite{Pouchol-Lorenzi, Perthame, P-al}. For any positive continuous initial condition $n^{0}$, the total population of tumour cells $\rho(t)$ converges to $\rho^\star:=\textstyle  \max (\frac{r}{d})$ as $t\to +\infty$.\\ This asymptotic cell population mass, which is its maximal value, is readily interpreted, as for all logistic models of tumour growth, as the tumour {\it carrying capacity}. Furthermore, the density $n(t,\cdot)$ viewed as a Radon measure supported on $[0,1]$ concentrates on the set 
\begin{equation}
A:=\left\{x\in[0,1],\;  r(x)-d(x) \rho^\star=0\right\} = \arg \max_{x \in [0,1]}\frac{r(x)}{d(x)}
\end{equation}
as $t\to +\infty$. If $A$ is reduced to a singleton $x^{\star}$, then in particular $n(t,\cdot) \rightharpoonup \rho^\star\delta_{x^\star}$
as $t\to +\infty$ in~$\mathcal{M}([0,1])$.

\subsubsection{{\it A priori} bounds}
We first indicate the derivation of an upper bound for $\rho$. Integrating the first equation of system \eqref{ide} with respect to $x$, we find using $\varphi \geq 0$:
\begin{align*}
\frac{d\rho}{dt}
&=\int_{0}^{1}\left[r(x)-d(x)\rho-\mu(x)\varphi(t)\right]n(t,x)\,dx 
\leq \displaystyle\max_{x\in [0,1]}\left(r(x) - d(x)\rho\right) \rho.
\end{align*} 
The right-hand side is negative as soon as $\displaystyle \max_{x\in [0,1]}\left(r(x) -d(x) \rho\right)< 0$, i.e., as soon as $\rho >  \textstyle \max \frac{r}{d}$.
Hence
\begin{equation}
\label{cond1}
\varlimsup_{t\to +\infty}
\rho(t) \leq  \max_{x\in [0,1]}\frac{r(x)}{d(x)} = \rho^\star. 
\end{equation}
Let us fix $y\in[0,1]$. 
Denoting $\omega^M(y):= \max_{\{x\in[0,1]\}} \omega(x,y)$, the previous bound yields
\begin{equation}
\label{cond2}
\forall y\in [0,1], \quad \varlimsup_{t\to +\infty}\chi(t,y)\leq  \omega^{M}(y)\overline{\rho}.
\end{equation}
Using the equation satisfied by $t \mapsto p(t,y)$, we find
\begin{equation}
\label{cond3}
 \varlimsup_{t\to +\infty} p(t,y)\leq \overline{p}(y):= \frac{ \omega^{M}(y) \rho^\star}{k_2}.
\end{equation}
Using the same arguments, one can prove that the population density $\ell$ is bounded from above. Indeed, 
\begin{align}
 \frac{d}{dt}  \ell (t,y)
= p(t,y)-\left(\nu(y)\rho(t)+k_1\right)\ell(t,y)
\leq p(t,y)-k_1\ell(t,y)
\end{align}
and we are led to 
\begin{equation}
\label{cond4}
\varlimsup_{t\to +\infty} \ell(t,y)\leq \overline{\ell}(y):=\frac{\overline{p}(y)}{k_1}.
\end{equation}
As a result, we obtain 
\begin{equation}
\varlimsup_{t\to +\infty}\varphi(t)\leq \overline{\varphi}:=\int_{0}^{1}\psi(y)\overline{\ell}(y)\,dy.
\end{equation}
Finally, we may argue as above for a lower bound for $\rho$ (on top of nonnegativity $\rho \geq 0$). Indeed, from 
\begin{align}
\frac{d\rho}{dt}
\geq \min_{x \in [0,1]} \left(r(x) -d(x)\rho -\mu(x)\varphi\right)\rho,
\end{align}
it follows that 
\begin{align}
\varliminf_{t\to +\infty} \rho(t)\geq \min_{x \in [0,1]} \left(\frac{r(x)-\mu(x)\overline{\varphi}}{d(x)}\right).
\end{align}
We accordingly consider an assumption leading to non-extinction, given by
\begin{equation}
\label{H}
\min_{x \in [0,1]} (r(x)-\mu(x)\overline{\varphi}) > 0,
\end{equation}
which depends only on the parameters $r$, $d$, $\mu$, $\psi$, $\omega$, $k_1$ and $k_2$. If this assumption is satisfied, eradication is impossible. 

\subsubsection{Asymptotics for the complete model, innate, non-adaptive response}\label{sec431}
This section is devoted to analysing the asymptotic behaviour of the model \eqref{ide} in the non-adaptive case, particularly represented by NK-lymphocytes rather than by T-lymphocytes, where $\varphi$ does not depend on $x$. 

As already mentioned in Section \ref{sec3}, assuming all functions to be constant, the IDE system has the ODE~\eqref{odesys} as a particular case. For that ODE, it has been proved that all three behaviours can occur: convergence to a (unique) trivial stable point (extinction or escape), convergence to a (unique) non-trivial stable point (equilibrium) and convergence to a limit cycle.  The existence of such periodic solutions means that there is no hope of deriving any unconditional result of convergence to steady states for the IDE model. 

In what follows, we prove a partial result, which makes the strong {\it a priori} assumption that $n$ converges. We then establish that the limit either equals $0$ or can precisely be characterised, see Theorem~\ref{thm_pricnipal}.
\begin{lemma}
\label{prop1}
Suppose that the density $n$ weakly converges in $\mathcal{M}([0,1])$, and denote $n^{\infty}$ the limit measure. Setting $\rho^{\infty}:= \int_{0}^{1}dn^{\infty}(x)$, and under the assumptions~\eqref{H4}-~\eqref{H3}, both densities $\ell$ and $p$ converge respectively to $\ell^{\infty}, p^{\infty} \in C^0([0,1])$ given by
	\begin{equation}
	\label{limit_pointwise}
	\begin{array}{ll}
	\ell^{\infty}(y)=\frac{p^{\infty}(y)}{\nu(y) \rho^{\infty}+k_1},& \\[0.3cm]
	p^{\infty}(y)=\frac{1}{k_2}\int_{0}^{1}\omega(x,y) \, dn^{\infty}(x).
	\end{array}%
	\end{equation} 	
\end{lemma}
\begin{proof} 
We let $y \in [0,1]$ be fixed. First remark that 
$\chi(t,y)$ converges to  $\bar{\chi}(y)$ given by
\begin{equation}
\bar{\chi}(y):=\int_{0}^{1}\omega(x,y) \,dn^{\infty}(x). 
\end{equation}
Hence $p(\cdot,y)$ satisfies a non-autonomous logistic ODE, given by 
	\begin{equation}
	\frac{dp(t,y)}{dt}=\left[\chi(t,y)-k_2p(t,y)\right]p(t,y).
	\end{equation}
For any given $\varepsilon>0$ and $t$ large enough (say $t \geq t_0$) such that $\chi(t,y)\leq\bar{\chi}(y)+\varepsilon$, we can write
	\begin{equation}
	\frac{dp(t,y)}{dt}\leq\left[\bar{\chi}(y)+\varepsilon-k_2p(t,y)\right]p(t,y),
	\end{equation}
$p$ is thus a sub-solution of the equation
	\begin{equation}
	\frac{du}{dt}(t)=\left[\bar{\chi}(y)+\varepsilon-k_2u(t)\right]u(t),
	\end{equation}
with initial condition chosen to be $u(t_0)=p(t_0,y)$. The solution of the latter logistic autonomous equation converges to $\textstyle \textstyle\frac{\bar{\chi}(y)+\varepsilon}{k_2}$ as $t\to +\infty$, since $p(t_0, y)>0$ by the assumption~\eqref{H4}. We conclude that 
\begin{equation}
\label{ineq1}
\forall \varepsilon>0, \quad \varlimsup_{t\to +\infty} p(t,y)\leq \lim\limits_{t\to +\infty}u(t)=\frac{\bar{\chi}(y)+\varepsilon}{k_2}. 
\end{equation}
Therefore, we may pass to the limit $\varepsilon\to 0$ in inequality \eqref{ineq1} to obtain  
\begin{equation}
\varlimsup_{t\to +\infty} p(t,y)\leq \frac{\bar{\chi}(y)}{k_2}. 
\end{equation}
Using a similar argument, we can obtain a bound from below, and then prove that
\begin{equation}
\forall y\in [0,1], \quad\lim_{t\to +\infty} p(t,y)= \frac{\bar{\chi}(y)}{k_2}=\frac{1}{k_2}\int_{0}^{1}\, \omega(x,y)\, dn^{\infty}(x)=p^{\infty}(y).
\end{equation}
Turning to the limit for $\ell$, we fix $y$ in $[0,1]$. Letting $L_y(t):= \ell(t,y)$, we have 
\begin{equation}
\frac{dL_{y}(t)}{dt}=A_{y}(t)-B_{y}(t)L_{y}(t),
\end{equation}
which is a non-autonomous linear differential equation, with 
\begin{equation}
\left\{
\begin{array}{ll}
\lim\limits_{t\to +\infty}A_{y}(t)=\lim\limits_{t\to \infty}p(t,y)=p^{\infty}(y)=:\bar{A}_{y},& \\[0.3cm]
\lim\limits_{t\to +\infty}\left(\nu(y)\rho(t)+k_1\right)=\nu(y)\rho^{\infty}+k_1=:\bar{B}_{y}. 
\end{array}%
\right.
\end{equation} 
For $\eps>0$ small enough (such that $\eps< \overline{B}_y$) and $t$ large enough (say $t \geq t_0$) such that $A_{y}(t)\leq\bar{A}_{y}+\eps$ and $B_{y}(t)\geq\bar{B}_{y}-\eps$, we can write
\begin{equation}
\frac{dL_{y}}{dt}\leq\left(\bar{A}_{y}+\eps\right)-\left(\bar{B}_{y}-\eps\right)L_{y},
\end{equation}
$L_{y}$ is thus a sub-solution of the autonomous equation given by
\begin{equation}
\frac{dv}{dt}=\left(\bar{A}_{y}+\eps\right)-\left(\bar{B}_{y}-\eps\right)v,
\end{equation}
with $v(t_0)=L_y(t_0)$, hence
\begin{equation}
\label{ineq}
\forall \eps>0, \quad \varlimsup_{t\to +\infty} L_{y}(t)\leq \lim\limits_{t\to +\infty}v(t)=\frac{\bar{A}_{y}+\eps}{\bar{B}_{y}-\eps}. 
\end{equation}
We then let  $\eps$ go to  $0$ to get
\begin{equation}
\forall y\in[0,1], \quad \varlimsup_{t\to +\infty} L_{y}(t)\leq \frac{\bar{A}_{y}}{\bar{B}_{y}}=\frac{p^{\infty}(y)}{\nu(y) \rho^{\infty}+k_1}. 
\end{equation}
Arguing in a similar manner to get a lower bound, we find
\begin{equation}
\forall y\in [0,1], \quad\lim_{t\to +\infty} \ell(t,y)= \frac{p^{\infty}(y)}{k_1+\nu(y)\rho^{\infty}}=\ell^{\infty}(y).
\end{equation}
\end{proof}
Let us now explain how to determine the only possible limits for the system, still making the strong {\it a priori} assumption that $n(t, \cdot)$ converges. We shall need a technical (but rather weak) assumption, namely
\begin{equation}
\label{H2}
\forall 0<\rho \leq \overline{\rho}, \;\forall 0<\varphi\leq \overline{\varphi}, \quad \arg\max_{x\in [0,1]}\left(r(x)-d(x)\rho-\mu(x)\varphi\right)=:\{x(\rho,\varphi)\}.
\end{equation}

\begin{remark}
One sufficient but more workable condition to have~\eqref{H2} is for the function $r-d \rho - \mu \varphi$ to be strictly concave over $[0,1]$, for all $0<\rho \leq \overline{\rho}$, $0<\varphi\leq \overline{\varphi}$, which assuming that all functions are smooth is verified as soon as $r'' - \rho d'' - \varphi \mu''<0$ on $[0,1]$ for all such values. 

We also note that~\eqref{H2} can be restricted to values $0<\rho \leq \overline{\rho}$, $0<\varphi\leq \overline{\varphi}$ such that the function $x \mapsto r(x)-d(x)\rho-\mu(x)\varphi$ has maximum zero, as the proof below shows.
\end{remark}
\begin{theorem}
\label{thm_pricnipal}
Suppose that the density $n$ weakly converges in $\mathcal{M}([0,1])$, and  denote $n^{\infty}$ the limit measure. Under the assumptions \eqref{H4}-\eqref{H3}-\eqref{H2}, then either $n^{\infty} = 0$ or $n^\infty$ is of the form  
\begin{equation*}
n^{\infty}=\rho^{\infty}\delta_{x^{\infty}},
\end{equation*}
where $x^{\infty}=x(\rho^{\infty},\varphi^{\infty})$ and $(\rho^{\infty},\varphi^{\infty})$ solves the following system over $(\rho,\varphi) \in \mathbb{R}^2$
\begin{equation}
\label{asymptotic}
\left\{
\begin{array}{ll}
\rho=\displaystyle\max_{x\in \lbrack 0,1]}\left(\frac{r(x)-\mu(x)\varphi}{d(x)}\right),& \\[0.3cm]
\varphi =\displaystyle \frac{\rho}{k_2}\int_{0}^{1}\frac{\psi(y)\,\omega(x(\rho,\varphi),y)}{\nu(y)\rho+k_1}\, dy.
\end{array}%
\right.
\end{equation} 
\end{theorem}
\begin{remark}
If one makes the additional assumption \eqref{H}, $\rho$ is bounded away from $0$ and hence we must have $n^\infty \neq 0$. In other words, the only possible limits are of the form given by the above result if~\eqref{H} holds. 
\end{remark}
\begin{proof}
We assume that $n^\infty \neq 0$. 
According to Lemma~\ref{prop1}, both $t\mapsto \ell(t,\cdot)$ and  $t\mapsto p(t,\cdot)$ converge pointwise to $\ell^\infty$ and $p^\infty$ implicitly given by formulae~\eqref{limit_pointwise}.

Let us justify that $\varphi$ converges. The bound~\eqref{cond4} shows that the function $(t,y) \mapsto  \psi(y) \ell (t,y)$ is dominated by the continuous function $y \mapsto \psi(y) \overline{\ell}(y)$, hence by the dominated convergence theorem, we have 
\begin{equation}
\label{lim_phi}
\lim\limits_{t\to +\infty}\varphi(t)=\varphi^{\infty}:=\int_{0}^{1} \psi(y) \ell^\infty(y) \,dy =\frac{1}{k_2}	\int_{0}^{1}\left[\frac{\psi(y)}{\nu(y)\rho^{\infty}+k_1}\int_{0}^{1}\omega(x,y)\, dn^{\infty}(x)\right] \,dy. 
\end{equation}  
For a fixed $x \in [0,1]$, $t \mapsto n(t, x)$ solves an exponential ODE (\textit{i.e}, of the form $z'(t) = a(t)z(t)$), whose time-dependent rate asymptotically approaches $r(x)-d(x)\rho^{\infty}-\mu(x)\varphi^{\infty}$. We may hence analyse its sign as follows.
\begin{itemize}
\item[i)] If  $r(x_0)-d(x_0)\rho^{\infty}-\mu(x_0) \varphi^{\infty}>0$ for some $x_0\in [0,1]$, then by continuity there exists an nontrivial interval $I\subset [0,1]$ containing $x_0$, along which $r(x)-d(x)\rho^{\infty}-\mu(x)\varphi^{\infty} \geq 2 \varepsilon$ for $\eps$ small enough. Since the convergence of $r(x)-d(x)\rho(t)-\mu(x)\varphi(t)$ towards $r(x)-d(x)\rho^{\infty}-\mu(x)\varphi^{\infty}$ as $t \to \infty$ is uniform with respect to $x \in [0,1]$, there exists $t_0>0$ such that $r(x)-d(x)\rho(t)-\mu(x)\varphi(t) \geq \varepsilon$ for all $t\geq t_0$ and $x\in I$.  Writing the solution \eqref{ide} in implicit form gives for all $t \geq 0$ 
\begin{equation*}
n(t,x)=n(t_0,x)\, e^{\int_{t_0}^{t}\left(r(x)-d(x)\rho(s)-\mu(x)\varphi(s)\right)\,ds},
\end{equation*}
which after integration over $[0,1]$ leads to
\begin{align*}
\rho(t) =\int_{0}^{1}n(t,x)\,dx \geq \int_{I}n(t_0,x) \, e^{\int_{t_0}^{t}\left(r(x)-d(x)\rho(s)-\mu(x)\varphi(s)\right)\,ds}\,dx
\geq |I| \inf_{x\in I}n(t_0,x) \, e^{\eps(t-t_0)}.
\end{align*}
with $|I|$ the Lebesgue measure of $I$.
Recalling the assumption \eqref{H4}, the continuous function $n(t_0,\cdot)$ is also positive, which shows that $\displaystyle \inf_{x\in I}n(t_0,x)>0$. Since the right-hand side goes to $+\infty$, we obtain a contradiction with the convergence of $\rho$.
\item[ii)] If $r -d\rho^{\infty}-\mu\varphi^{\infty}<0$ on the whole of $[0,1]$, one readily proves that $\rho$ converges to $0$ which is incompatible with the convergence of $\rho$ to a positive limit (since $n^\infty \neq 0$).
\end{itemize} 
The function $r-d\rho^{\infty}-\mu\varphi^{\infty}$  is thus non positive on $[0,1]$, and its maximum equals $0$. 
This is equivalent to saying that $\rho^\infty =\textstyle \max (\frac{r-\mu\varphi^\infty}{d})$.

Assumption \eqref{H2} ensures that the maximum point $x^\infty:= x(\rho^{\infty},\varphi^{\infty})$ is unique. 
Furthermore, the first statement i) further shows that $n(t, x)$ vanishes at any other point $x$ than $x \neq x^\infty$. 
We have thus proved that $n$ concentrates at $x^\infty$, hence $n^\infty = \rho^\infty \delta_{x^\infty}$.

Finally, inserting $n^\infty = \rho^\infty \delta_{x^\infty}$ into the formula~\eqref{lim_phi}, we obtain the second equation, concluding the proof.
\begin{remark}
In general, there is no close formula for the solutions of~\eqref{asymptotic}, which may not be unique. In practice, this system may be solved numerically by any method aiming at finding fixed points of the underlying mapping. Hence, assuming convergence of $n$, this theorem does provide a rather complete picture of the possible non-trivial limits the system may reach. When there exists a unique solution to~\eqref{asymptotic}, a single such limit is characterised.
\end{remark}
\end{proof}

\subsubsection{Asymptotics in the adaptive and mixed innate-adaptive response case}
We now sketch the extension of Theorem~\ref{thm_pricnipal} to the (more general) case where $\varphi$ depends on $x$. In this case, we may obtain a result similar to Theorem~\ref{thm_pricnipal}, but at the expense of an assumption stronger than~\eqref{H2} and a more intricate system solved by the stationary state. 

Indeed, keeping the same notations, we make the assumption that 
for all $0<  \rho \leq \overline{\rho}$ and for all functions $\ell \in \mathcal{C}([0,1])$ satisfying $0 \leq \ell(y)\leq \overline{\ell}(y)$ for all $y \in [0,1]$,
\begin{equation}
\label{H2-tilde}
\arg\max_{x\in [0,1]}\bigg(r(x)-d(x)\rho-\mu(x) \int_0^1 \Psi(x,y) \ell(y)\,dy  \bigg)=:\{x(\rho,\ell)\}.
\end{equation}
Following the proof of Theorem~\ref{thm_pricnipal}, one can then prove in exactly the same way:
\begin{theorem}
\label{thm_secondary}
Under the assumptions \eqref{H4}-\eqref{H3}-\eqref{H2-tilde}, supposing that $n$ converges weakly in $\mathcal{M}([0,1])$ to some $n^\infty$, then either $n^\infty = 0$ or $n^{\infty}$ is of the form  
\begin{equation*}
n^{\infty}=\rho^{\infty}\delta_{x^{\infty}},
\end{equation*}
where $x^{\infty}=x(\rho^{\infty},\ell^{\infty})$ and $(\rho^{\infty},\ell^{\infty})$ solves the following system over $(\rho,\ell) \in \mathbb{R} \times \mathcal{C}([0,1])$
\begin{equation}
\label{syst}
\left\{
\begin{array}{ll}
\rho=\displaystyle\max_{x\in \lbrack 0,1]}\left(\frac{r(x)-\mu(x)\int_0^1 \Psi(x,y) \ell(y)\,dy}{d(x)}\right),& \\[0.3cm]
\ell(y) =\displaystyle \frac{\rho}{k_2}\frac{\omega(x(\rho,\ell),y)}{\nu(y)\rho+k_1}, \quad y \in [0,1].
\end{array}%
\right.
\end{equation} 
\end{theorem}
\section{Numerical simulations}\label{sec5}
In this section, we present some numerical simulations of system~\eqref{ID}.

\subsection{Numerical approach and parameters}
We follow the numerical method given in~\cite{Pouchol_stage} and we select a discretisation of phenotype intervals $[0,1]$ consisting of $1000$ points for the computational domain of the independent variables $x$ and $y$ and let $t\in [0,T]$, unless otherwise specified, we choose the final time to be $T=100$ or $T=1000$. Time-discretisation is made with time step equal to $0.1$ except for the extensive simulations of Figure~\ref{heatmaps} where it is set equal to $1$.

The function $\Psi = \Psi_\lambda$ underlying $\varphi$ in~\eqref{eq-varphi} is chosen to be of the form~\eqref{combary}, and the kernel $\omega$ underlying~$\chi$ in~\eqref{eq-chi} is given by $\omega(x,y) = \alpha \frac{1}{s}e^{-|x-y|/s}$. Parameters $\lambda$ and $v$ (for $\Psi$), and $s$ (for $\omega$) may vary across simulations.
Unless otherwise specified, all other parameters take values as presented in Table~\ref{tab1}. 

We emphasise that parameters have been chosen arbitrarily in the absence of suitable experimental data, in order to reproduce different biological scenarios.

To define the initial density of tumour cells, we use a Gaussian profile, and a homogeneous condition for competent immune cells $\ell$, while the na\"ive immune cells $p$ are distributed over the whole interval $[0,1]$:
\begin{equation}
\label{initialdata}
\left\{
\begin{array}{ll}
n^0(x)=n(0,x)=\frac{C}{\sqrt{2\pi \sigma_0^2}}\exp(\frac{-(x-m)^2}{2e^2}), &\\ [0.3cm] \ell^0(y)=\ell(0,y)=0, &\\ [0.3cm]
p^0(y)=p(0,y)=1-y^2,
\end{array}%
\right.
\end{equation}
with $m=0.5, e=0.1$, and a normalisation constant $C>0$ chosen so that $\rho(0)=1$. Thus, we start with a total tumour cell mass equal to $1$, and the phenotype $x$ is initially concentrated at $0.5$.

We thus assume that activated NK-cells and competent T-cells $\ell(t,y)$ are absent at time $t=0$, and that the most aggressive inactive NK-cells and na\"ive T-cells $p(t,y)$ have been duly informed by circulating NK-cells and by APCs and present themselves at time $t=0$ at the tumour site to activate the immune response by $\ell(t,y)$ cells. 
\begin{center}
\begin{tabular}{l|c|c}
\hline
Parameter/function & Biological meaning&Value\\
\hline
$r(x)$ & Proliferation rate of tumour cells&$0.666-0.132x^2$\\
$d(x)$ & Death rate of tumour cells& $0.5(1-0.3x)$\\
$\mu(x)$& Sensitivity to the effects of the immune response&$1-0.1x^2$\\ 
$\psi(y)$& Efficacy of the immune response&$0.5y^2$\\
$\nu(y)$& Immunotolerance of immune cells induced by tumour cells&$0.5-0.1y$ (in Section~\ref{sec5.2})\\
 & &$1-0.1y$ (in Section~\ref{sec5.3})\\
$k_1$& Natural death rate of competent immune cells& $0.5$ (in Section~\ref{sec5.2})\\
 & & $0.01$ (in Section~\ref{sec5.3})\\
$k_2$& Strength of logistic death rate of na\"ive immune cells&$1.5$\\
$\alpha$& Strength of the immune response&$1$\\
$h$& Strength of treatment with ICIs&$10$ \\
\hline
\end{tabular}
\captionof{table}{Values of the model parameters/functions used to carry out numerical simulations.}
\label{tab1}
\end{center}

 \subsection{Tumour development in the absence of the immune response} We begin by establishing a baseline scenario in which tumour cells proliferate and die according to the modelling approach described in Section \ref{sec3}, i.e., in the absence of the immune response, logistic growth of the tumour cell population. 
 \begin{figure}[H]
    \centering{
\includegraphics[width=0.6\textwidth]{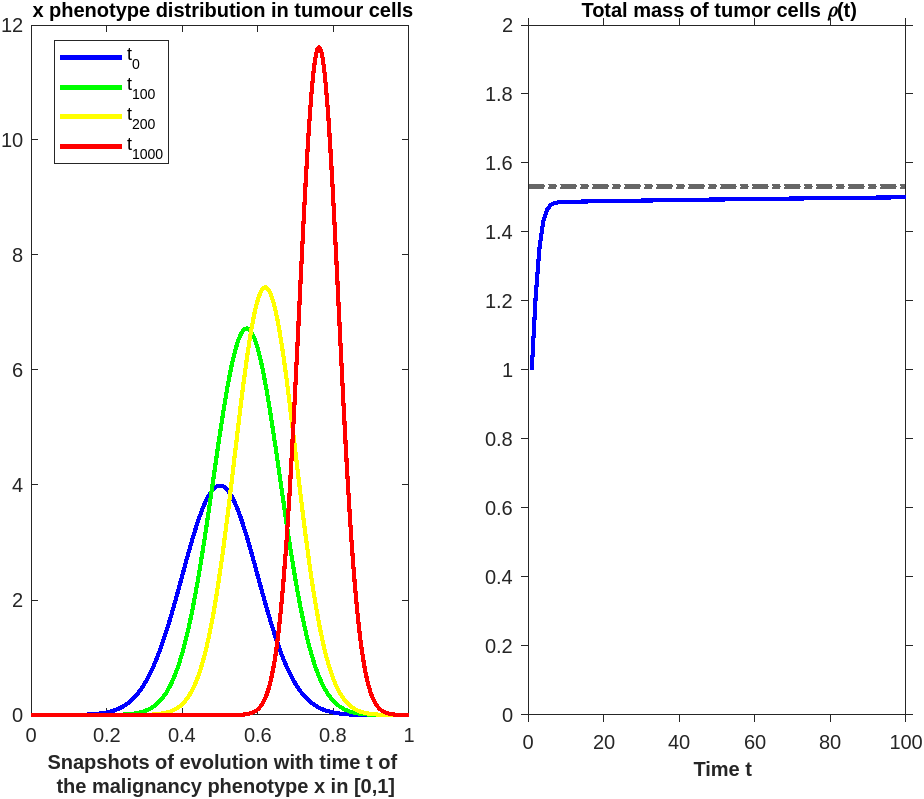}}\vspace{.3cm}
\caption{Numerical simulation of the solution to~(\ref{singleide}) ({\it complete absence of immune response}). {\bf Left panel}, plots of cell densities $n(t,\cdot)$ at different times up to $T=1000$ (in red): the phenotype $x$ evolves towards more and more malignancy. {\bf Right panel}, initial dynamics of the total mass of tumour cells $\rho(t)$ for $t$ between $0$ and $100$. The black dashed line highlights a numerical estimation of the tumour cell carrying capacity $\rho^\star$ and the parameter values are as listed on Table~\ref{tab1}, with $\rho(0)=1$.}
\label{fig:single}
\end{figure}
According to Section~\ref{sec4}, we expect convergence of $\rho$ and weak convergence of $n$ to a weighted Dirac mass. Moreover, the limit for $\rho$ is $\rho^\star=\textstyle \max (\frac{r}{d})$, which corresponds to the carrying capacity of the tumour, i.e., the saturation term reached by the total mass of tumour cells due to within-population competition for space and resources. Here, $\rho^\star\approx 1.53$ and this is what we observe on Figure~\ref{fig:single} to the right. On the other hand, the phenotype at which the density concentrates is located at $\arg \max(\textstyle \frac{r}{d}) = \{x^\star\}$ with $x^\star \approx 0.86$, which becomes apparent on Figure~\ref{fig:single} to the left (and would be seen even more clearly if the simulations were run longer).

As already mentioned in the introduction, we will from now on, when the immune response is activated, interpret solutions for which the total mass of tumour cells approaches this carrying capacity $\rho^\star$ as ``tumour escape''. This represents one case of the three $E$s in which the immune cells are present at the tumour site but are inefficient in interacting with the tumour cells.

\subsection{Simulations in the mixed innate-adaptive case ($0<\lambda<1$), no treatment} \label{sec5.2}
We have explored in simulations separately the innate, adaptive and mixed innate-adaptive cases, which all can lead to the three $E$s. No treatment with {\it ICI}s is considered here, ensured by setting $ICI = 0$. In agreement with Theorems~\ref{thm_pricnipal} and \ref{thm_secondary}, if convergence of the density of cancer cells occurs, we find that tumour cells asymptotically concentrate on a single phenotype, and total masses of cells all converge. Furthermore, the phenotype on which the cancer cell density concentrates as well as the asymptotic masses of cells have been checked to match the specific values uncovered by Theorems~\ref{thm_pricnipal} and \ref{thm_secondary}.  

To avoid repetitive figures, we have chosen to focus mostly on the mixed innate-adaptive case: Figures~\ref{fig:eradication3}, \ref{fig:equilibrium3} and \ref{fig:escape3} illustrate different possible asymptotic behaviours with $\lambda=0.5$. Let us mention that simulations run with $\lambda = 1$ lead to qualitatively similar results.

For simulations illustrated on Figures \ref{fig:eradication3}-\ref{fig:equilibrium3}-\ref{fig:escape3}:
\begin{itemize}
\item[]\textbf{Upper row}. Evolution in time $t$ of the densities $x \mapsto n(t,x)$ (left panel); $y \mapsto \ell(t,y)$ (central panel), and $y \mapsto p(t,y)$ (right panel), with the initial conditions in blue, and the final ones in red.
\item[]\textbf{Lower row}. Initial time dynamics of the total mass of tumour cells $\rho(t)$ (left panel), of the total mass of competent immune cells $\int_0^1\ell(t,y)\, dy$ (central panel), and of the total mass of na\"ive immune cells $\int_0^1 p(t,y)\, dy$ (right panel). 
\end{itemize}
We insist that the final time shown for densities and total numbers of cells might not be the same in a given plot. This is because total numbers of cells rapidly reach equilibrium while densities may converge slowly to their limit, at least when the limit is a Dirac mass.

\paragraph{Results.} When the parameter $s$ is small enough, and for all considered values of the parameter $v$, the total mass of tumour cells decreases steadily over time until the tumour cell population is completely eradicated. This is due to the fact that precise detection and transmission of the malignancy phenotype $x$ by circulating NK-cells and by APCs (i.e., small values of the parameter $s$ in the function $\chi(t,y)$) promotes the eradication of tumour cells by CD8$+$ T-cells.

Eradication also occurs for larger values of $s$ (here $s=1$) as long as $v$ is small enough, see Figure~\ref{fig:eradication3}. In fact, numerical results not displayed here show that for the same value for $s$ but with an innate immune response ($\lambda = 0$), one can obtain escape rather than eradication, highlighting the importance of adaptive immune responses.

Fixing the value $s=1$, the results displayed on Figure~\ref{fig:equilibrium3} show that intermediate values of the parameter~$v$ (which measures the precision of the targeting of cancer cells by the immune response) facilitate the coexistence between tumour and immune CD8+ T-lymphocytes, while the total mass of tumour cells remains at a low level. 

Finally, still with $s=1$, Figure \ref{fig:escape3} shows that a large value of the parameter $v$ leads to tumour escape. Taken together, these results suggest the idea that the efficacy of the anti-tumour immune response is affected by the specificity of the anti-tumour immune response and also by the specificity of the message transmitted by circulating NK-cells to inactive NK-cells and by APCs to na\"ive T cells. Additionally, high values of parameters $s$ and $v$, respectively, are associated with low masses of immune cells and less effective immune response, which may enhance tumour development. 
 
\begin{figure}[h]
	\centering{
  \caption*{\bf{Dynamics of tumour cells, effector and naïve lymphocytes with $(s,v)=(1,0.1)$}}		
  \includegraphics[width=0.8\textwidth]{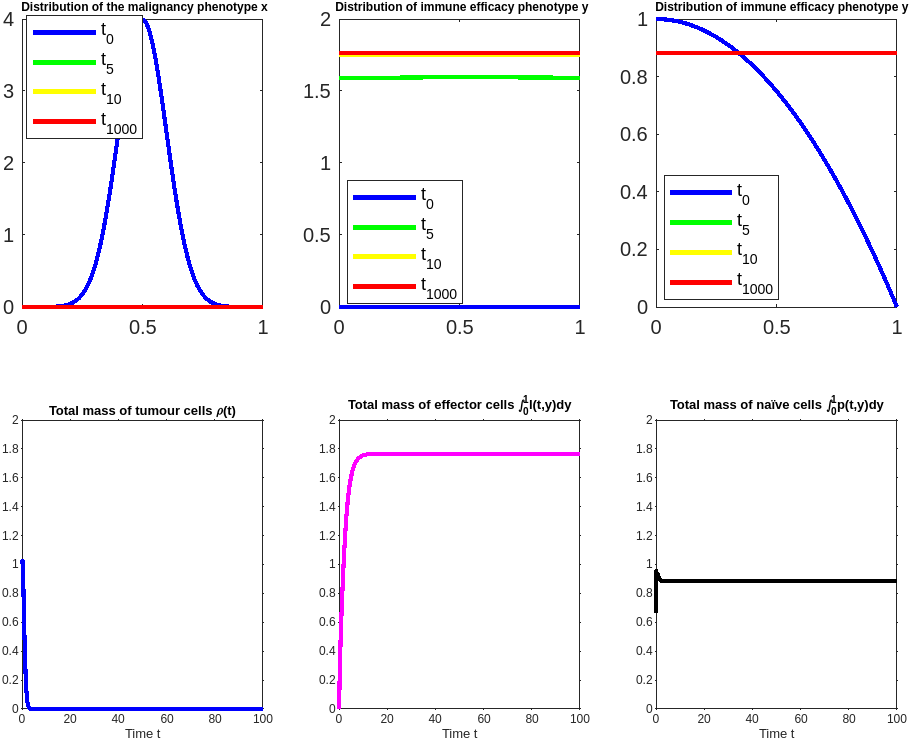}
		\caption{{\small \label{fig:eradication3} {\bf Eradication.} Mixed innate/adaptive case ($\lambda=0.5$). Numerical simulations of (\ref{ide}) with $(s,v)=(1,0.1)$ at different times up to $T=1000$ in red ({\bf{upper row}}) and initial evolution with time $t$ from $0$ to $100$ of the total masses of cells $\int_{0}^{1}n(t,x)\,dx, \int_{0}^{1}\ell(t,y)\,dy,\int_{0}^{1}p(t,y)\,dy$. As one can see on the right panels, eradication comes quickly with this highly precise value of the targeting parameter $v$.}}}
\end{figure}

\begin{figure}[h]
	\centering{
 \caption*{{\bf Dynamics of tumour cells, effector and naïve lymphocytes with $(s,v)=(1,0.5)$}}
		\includegraphics[width=0.8\textwidth]{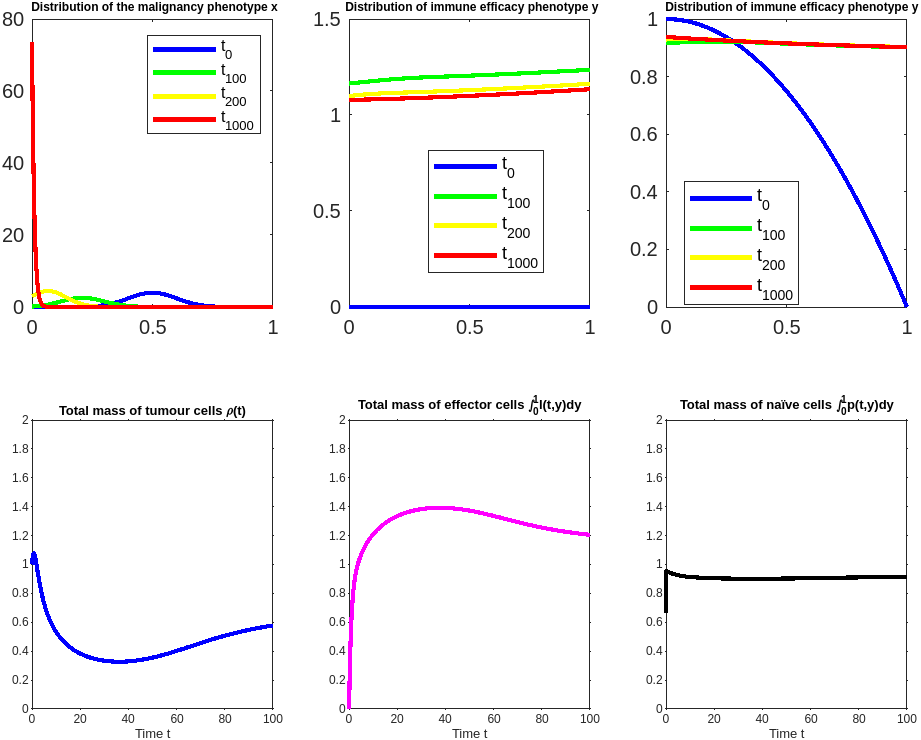}
		\caption{{\small \label{fig:equilibrium3} {\bf Equilibrium.} Mixed innate/adaptive case ($\lambda=0.5$). Numerical simulations of (\ref{ide}) with $(s,v)=(1,0.5)$, plots of the population densities $n, \ell ,p$ at different times up to $T=1000$ in red ({\bf{upper row}}) and initial evolution with time $t$ from $0$ to $100$ of the total masses of cells $\int_{0}^{1}n(t,x)\,dx, \int_{0}^{1}\ell(t,y)\,dy,\int_{0}^{1}p(t,y)\,dy$ ({\bf{lower row}}). Note that the total mass of tumour cells stabilises at an intermediate value between extinction and the tumour carrying capacity, and that the malignancy phenotype $x$ concentrates on the phenotype $0$.}}}
		
\end{figure}

\begin{figure}[h]
	\centering{
  \caption*{\bf{Dynamics of tumour cells, effector and naïve lymphocytes with $(s,v)=(1,1)$}}		
  \includegraphics[width=0.8\textwidth]{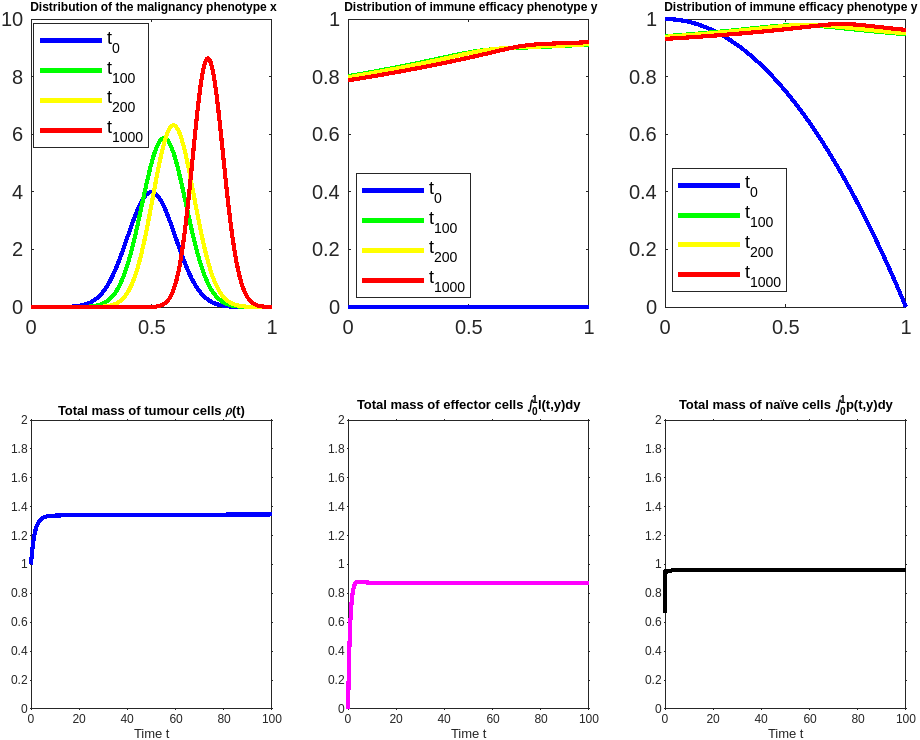}
		\caption{{\small \label{fig:escape3} {\bf Escape.} Mixed innate/adaptive case ($\lambda=0.5$). Numerical simulations of (\ref{ide}) with $(s,v)=(1,1)$, plots of the population densities $n, \ell ,p$ at different times up to $T=1000$ in red ({\bf{upper row}}) and initial evolution with time $t$ from $0$ to $100$ of the total masses  $\int_{0}^{1}n(t,x)\,dx, \int_{0}^{1}\ell(t,y)\,dy,\int_{0}^{1}p(t,y)\,dy$  ({\bf{lower row}}). Note that the total mass of tumour cells stabilises to a value close to the tumour carrying capacity $\rho^\star$ and that the malignancy phenotype $x$ concentrates around a phenotype close to $1$.}}}
		
\end{figure}
As shown on Figures \ref{fig:eradication3}-\ref{fig:escape3}, for an intermediate value of the parameter $\lambda$ in $[0,1]$ and a rather imprecise value of the detection parameter $s$, here fixed at $s=1$, we can see that for low values of the targeting parameter $v$ (i.e., high precision of targeting), the specific anti-tumour immune response involving CD8+ T-lymphocytes is relatively high, leading to the total eradication of tumour cells; that intermediate values of $v$ lead to a co-existence state representing tumour-immune response equilibrium; and finally, that high values of $v$ (i.e., poor precision of targeting) decrease the efficacy of the specific immune response.

This is also illustrated on Figure~\ref{heatmaps} by a heatmap representation, where we plot the asymptotic tumour mass relative to the tumour carrying capacity (i.e.,$\tfrac{\rho^\infty}{\rho^\star}$) as a function of the parameter $(s,v)$, with a mixed innate/adaptive immune response ($\lambda=0.5$). The value $\rho^\infty$ is estimated by the value $\rho(T)$ with $T = 500$.

One can in particular see (left part of the heatmap figure) that 
whatever the value of the detection parameter $v$ is, the model still yields low values of the total mass of tumour cells, i.e., efficacy of the immune response, provided that the precision of targeting is high enough (i.e., $s$ is small enough). On the right lower part corresponding to $s$ and $v$ close to $1$, escape is the outcome, which corresponds to the immune system failing to control the tumour. 

\paragraph{Non-extinction condition.}
Part of these results can be interpreted in light of our non-extinction condition~\eqref{H}, which leaves only equilibrium and escape as possible outcomes. The corresponding assumption in the adaptive case \vspace{-.2cm}reads 
\begin{equation*}
\min_{x \in [0,1]} (r(x)-\mu(x)\overline{\varphi}(x)) > 0, \qquad \overline{\varphi}(x) = \frac{\rho^\star}{k_1 k_2} \int_0^1 \Psi(x,y) \, \omega^M(y) \,dy. 
\end{equation*}
Hence, one can see that if $\overline{\varphi}$ is sufficiently small on the whole of $[0,1]$, this condition is satisfied. Given our choice of functions $\Psi$, $\omega$, we have $w^M = \tfrac{\alpha}{s}$, which leads to 
\begin{equation*}
\overline{\varphi}(x) = \frac{1}{s} \frac{\alpha \rho^\star}{k_1 k_2} \int_0^1 \big((1-\lambda) + \lambda\frac{1}{v} e^{-|x-y|/v}\big) \Psi(y)  \,dy. 
\end{equation*}
Consequently, the non-extinction condition will become valid for $s$ sufficiently large. This observation is coherent with the results obtained above. In view of Figure~\ref{fig:eradication3}, $s=1$ is not large enough a value for the non-extinction condition to be satisfied when $v=0.1$, since one obtains eradication in this case.  
Note that if $\lambda = 1$,  $v$ sufficiently large also enforces the non-extinction condition, making extinction impossible.

\paragraph{Asymptotic reduction to the ODE~(\ref{odesys}).}
As can be seen for most simulations, both immune cell densities converge to equilibrium distributions that are close to constant functions, while the cancer cell density either converges to $0$ or to a Dirac mass. Asymptotically, one can hence approximate the IDE system by the ODE system~\eqref{odesys}. More precisely, denoting $x^\star$ the phenotype on which the cancer cell density concentrates, $t \mapsto (\rho(t), \int_0^1 \ell(t, y)\,dy, \int_0^1 p(t, y)\,dy)$ is asymptotically close to solving~\eqref{odesys} with parameter values $r = r(x^\star)$, $d = d(x^\star)$, $\mu = \mu(x^\star) \int_0^1 \Psi(x^\star, y)\,dy$, $\nu = \int_0^1 \nu(y)\,dy$,  $\omega = \int_0^1 \omega(x^\star, y)\,dy$. This result is obtained by integrating each equation of~\eqref{ide} with respect to the structuring variable.
\begin{figure}[H]
	\centering{
		\includegraphics[width=0.9\textwidth]{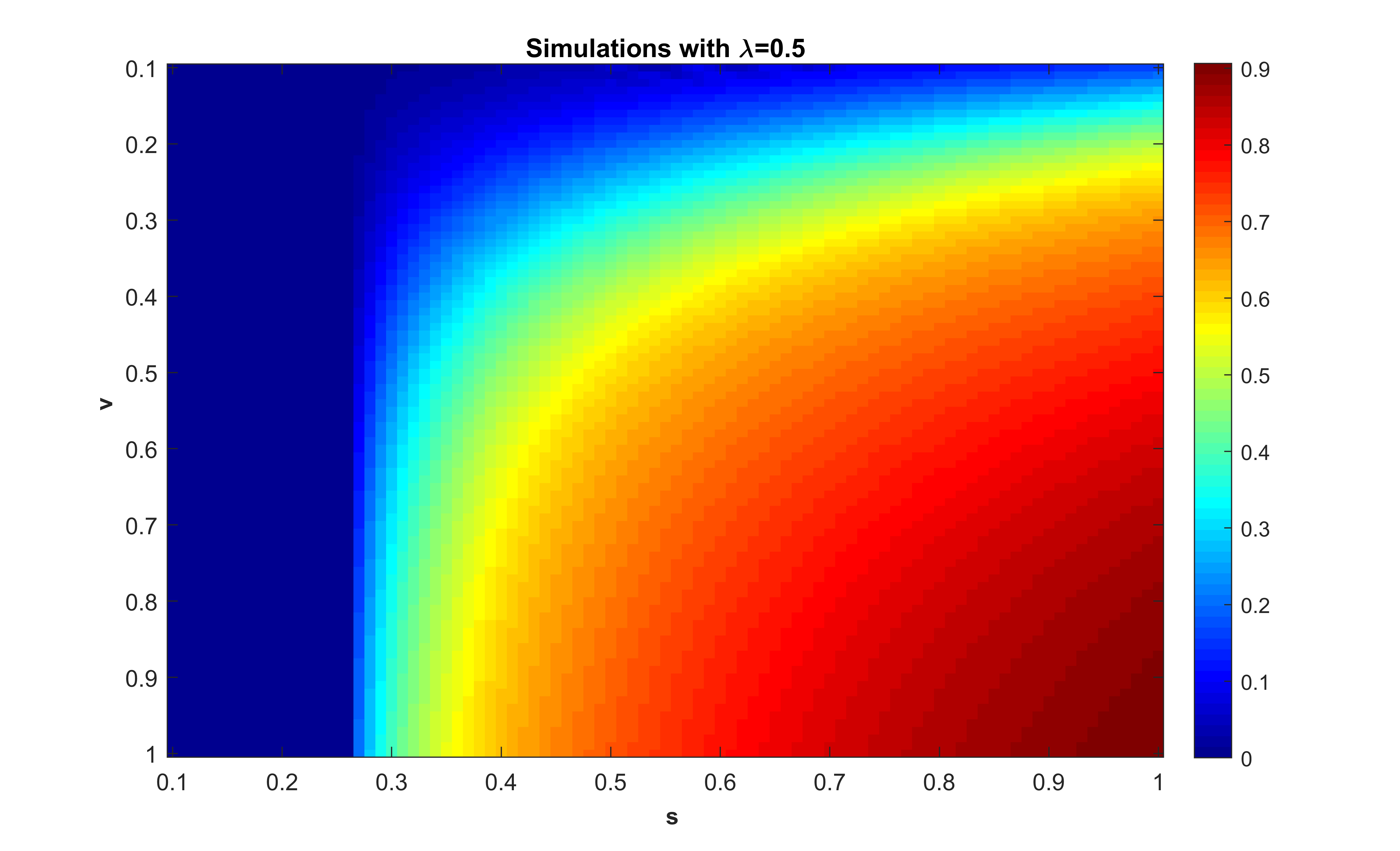}
		\caption {{\small \label{heatmaps}{{\bf Precision of detection parameter $s$ vs. precision of targeting parameter $v$.} (Here again, $\lambda=0.5$). Heatmap representation of the contribution of the two localisation kernel parameters $s$ and $v$ to the relative
density $\frac{\rho^{\infty}}{\rho^{\star}}$ of total tumour cells at the end of simulations ($T=500$), in the case without treatment. Recall that $\rho^{\star}$ is the maximum possible density (the carrying capacity) of tumour cells. Of note, as already mentioned above, one can see that, provided that the precision of detection is high enough (low values of the parameter $s$, between $0.1$ and $0.25$), whatever the precision of the targeting parameter $v$ in a wide range between $0.1$ and $1$, the immune response yields extinction or quasi-extinction of the tumour mass (deep blue rectangular zone on the left). Conversely, for values of poor precision of $s$ and $v$, here $(s,v)$ in a neighbourhood of $(1,1)$, one can see that anti-tumour efficacy of the immune response is poor.}}}
		}
\end{figure}

Taken together, the numerical results that we have presented in the previous subsections suggest that the model has validity for providing a consistent qualitative description of the anti-tumour immune response involving both NK cells and CD8+ T-lymphocytes. 

\subsection{From escape to eradication with $ICI$s, strictly adaptive response ($\lambda=1$)\label{sec5.3}}

Starting from a situation in which we have tumour escape without $ICI$s, we show how introducing them may lead the tumour cell population to equilibrium, and by increasing the drug dose, to elimination. 
We consider the strictly adaptive case where $\Psi$ is given by~\eqref{adaptative-function}, i.e., $\Psi_\lambda$ with $\lambda=1$ (T-cells only).

We fix $(s,v)=(1,2)$, the other parameter values being as listed in Table~\ref{tab1}. The drug dose is successively set to $ICI = 0$, to $ICI =1$ and to $ICI = 10$, with $h=10$. As Figure~\ref{control} shows, these choices lead to escape, equilibrium and eradication, respectively. Escape is associated with a distribution of the malignancy phenotype moving to the right (with eventual concentration towards a Dirac mass located on the right), while equilibrium is associated with a phenotype remaining at around $x=0.5$, and eradication occurs without any visible density shift to either side (eventually yielding $0$ - formally, as this value corresponds to a vanished cell population). This last point, in apparent contradiction with the situation illustrated on Figure~\ref{fig:equilibrium3}, where a clear shift to the left is apparent, may be interpreted in light of the high value of parameter $v$ in the present case, i.e., of a low precision of the targeted immune response.
\begin{figure}[H]
  \centering
  \begin{subfigure}{0.49\linewidth}
    \centering
    \includegraphics[width=.9\linewidth]{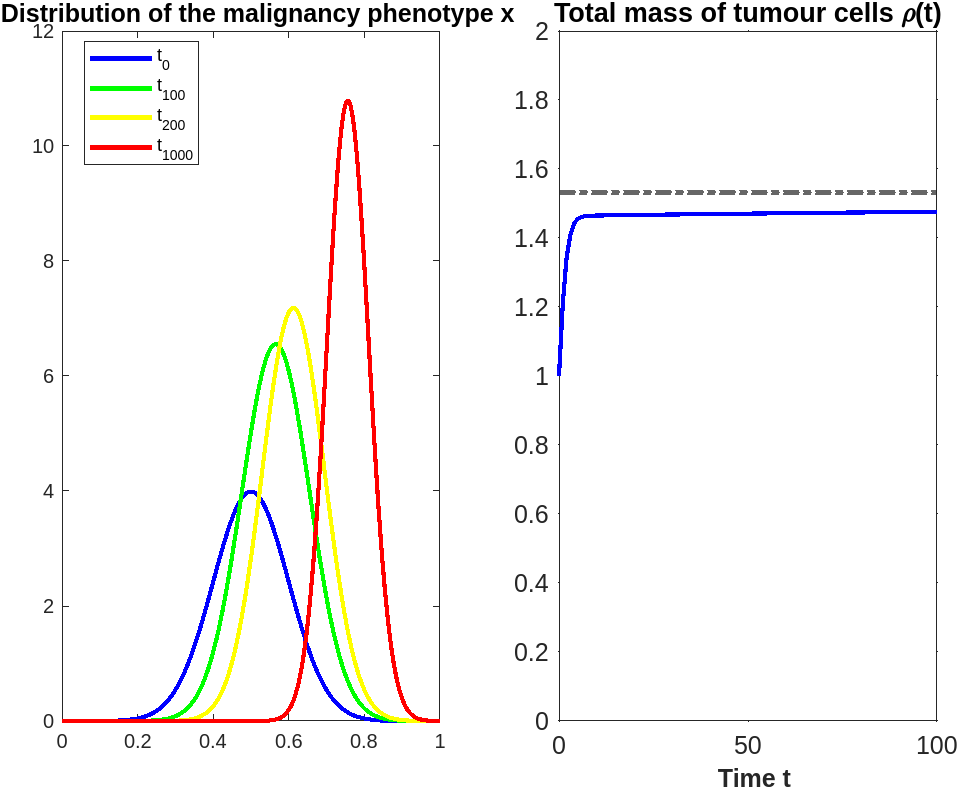}
    \caption{{\bf Escape.} Simulation with $ICI=0$.}
  \end{subfigure}
 \begin{subfigure}{0.49\linewidth}
    \centering
    \includegraphics[width=.9\linewidth]{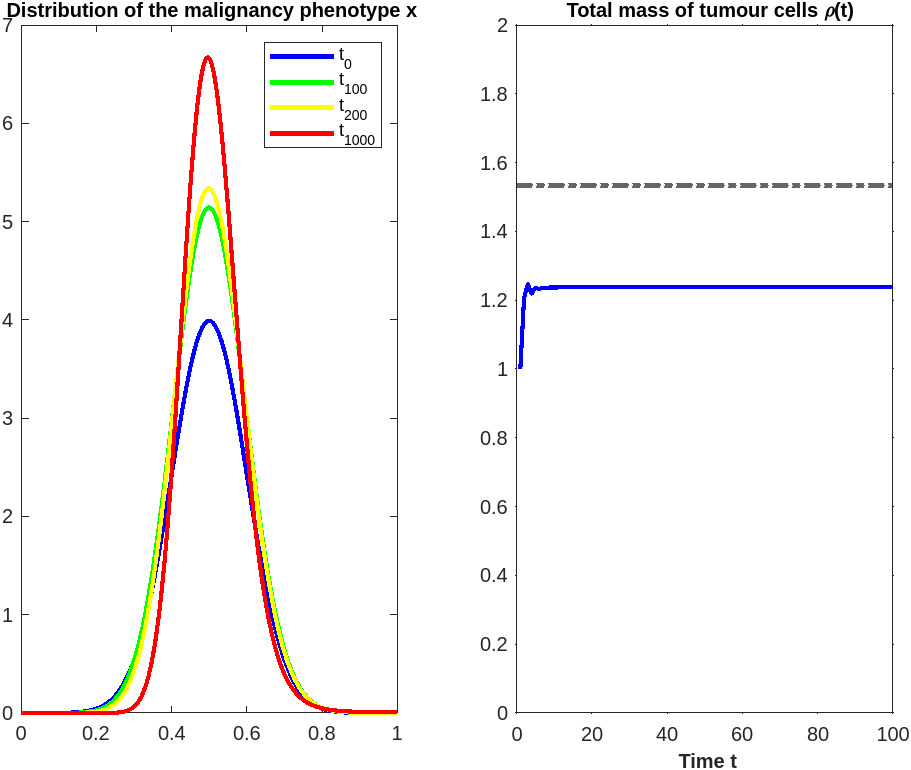}
    \caption{{\bf Equilibrium.} Simulation with $ICI=1$.}
  \end{subfigure} 
  
  \begin{subfigure}{0.49\linewidth}
    \centering
    \includegraphics[width=.9\linewidth]{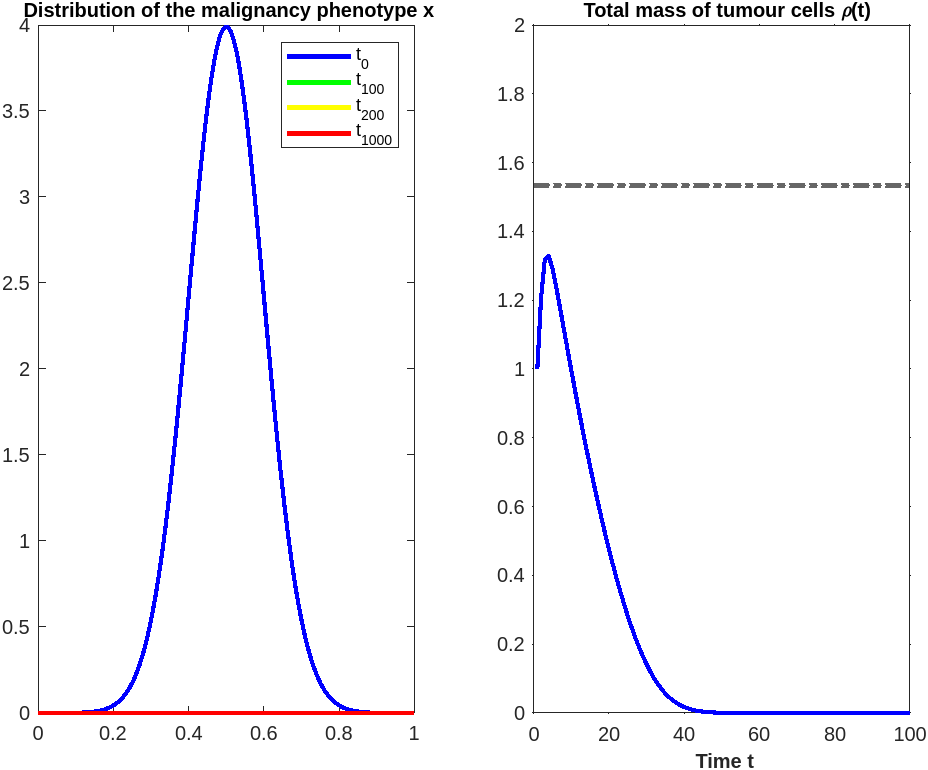}
    \caption{{\bf Eradication.} Simulation with $ICI=10$.}
  \end{subfigure}  
  \caption{Numerical simulations of (\ref{ID}) with $\lambda=1$ (T-cells only), $(s,v)=(1,2)$, and increasing levels of $ICI$s. Plots of the population density of tumour cells $n$ at different times up to $T=1000$ in red ({\bf{left panel}}) and initial evolution with time $t$ from $0$ to $100$ of the total mass $\int_{0}^{1}n(t,x)\,dx$  ({\bf{right panel}}). The black dashed line stands for the tumour carrying capacity $\rho^\star$.}  
  \label{control}
\end{figure}

 \subsection{Periodic solutions} We can also numerically address the existence of periodic solutions. We first take all the parameters and functions to be equal to those chosen for the ODE model in the periodic case, i.e., leading to Figure~\ref{odefig}. Then, we perturb them by a small parameter $0<\delta\ll 1$. In this case, an oscillatory behaviour also emerges, corresponding to a co-existence state representing a time-dependent periodic solution, see Figure~\ref{periodic}. We have not been able to analytically address the existence of periodic solutions, except for the very specific case where all functions are constant, in which case we recover the model~\eqref{odesys}, for which we know that periodic solutions do exist~\cite{Kaidetal}. 
\begin{figure}[H]
	\centering{
		\includegraphics[width=.65\linewidth]{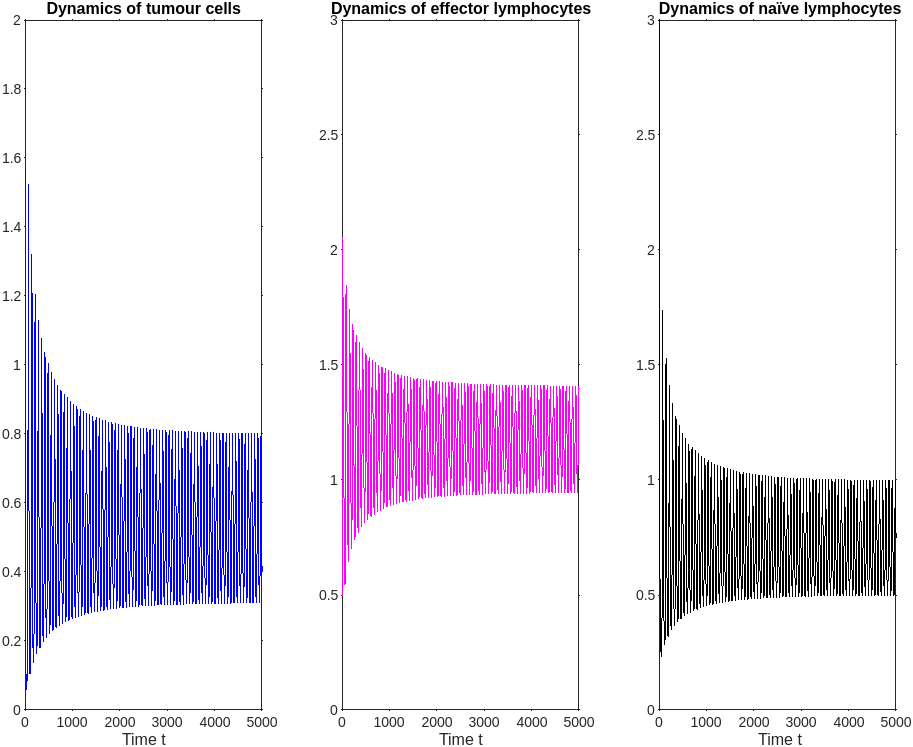}
		\caption{{\small \label{periodic}{Evolution with time of the tumour total density $\rho(t)$ (in blue), competent T cells total density $\sigma(t) = \int_0^1\ell(t,y)\,dy$ (in magenta), and the total density of na\"ive T cells $\int_0^ 1p(t,y)\,dy$ (in black) for $T=5000$.}}}}
\end{figure}\vspace{-0.2cm}

\section{Conclusions and research perspectives}\label{sec6}
\subsection{Summary of the mathematical results}
We have proposed a new mathematical model of tumour-immune interactions in which cell populations are structured by continuous phenotype variables representing their aggressiveness. Despite its simplicity, our model features some relevant phenomena, and it captures the three $E$s of immunoediting - eradication, equilibrium and escape. In particular, it reproduces the formation of an equilibrium, which characterises the capacity of the immune system to contain tumour growth.

In Section \ref{sec4}, we showed through an asymptotic analysis of the model that under the {\it a priori} assumption that the population of tumour cells converges to a certain measure, such a measure can precisely be characterised when it is not the trivial measure. 

We explained why convergence cannot be the general outcome: our model does have the ODE system~\eqref{odesys}, with known possible periodic behaviour, as a particular case.
Finding which parameters lead to convergence or to oscillatory behaviours is a completely open question.

Our model can incorporate three different types of anti-tumour immune responses: innate, adaptive, and a combination of both immune responses. By numerically comparing these three cases in Section \ref{sec5}, the outcomes are as follows:

\paragraph{{\bf{Innate anti-tumour immune response.}}} If the parameter $s$, which determines how localised the phenotype $x$ is with respect to the phenotype $y$, is small enough, then the tumour is always eliminated. For intermediate values of $s$, we obtain convergence to a limit coherent with Theorem \ref{thm_pricnipal}: a coexistence state occurs, yielding a persistent tumour cell population at a controlled level. Finally, high values of the parameter $s$ reduce the efficacy of the anti-tumour immune response and lead to tumour escape. For particular choices of the model parameters, the numerical results also show periodic solutions, characterised by periodic alternating growth and decay of all the immune and tumour cell populations. 

\paragraph{{\bf{Adaptive anti-tumour immune response.}}}
The situations that we have numerically explored in the adaptive anti-tumour response, showed that both the specificity of the response of competent immune cells (i.e., the parameter~$v$) and the specificity of the message transmitted by APCs (i.e., the parameter~$s$) play a key role in the tumour-immune interactions. In fact, when $s$ and $v$ are both small, our results indicate that tumour eradication can occur, while higher values of $s$ or $v$ may result in tumour escape.

\paragraph{{\bf{Combination of the innate and the adaptive anti-tumour immune responses.}}}Increasing the specificity of the adaptive immune response (low values of the parameter $v$) has a beneficial effect on the immune response to tumours, whereas higher values of the parameter $v$ can be detrimental to the anti-tumour immune action.

\paragraph{{\bf{Simulations of the effect of constant drug doses.}}}
Our numerical simulations show that a constant control allows to maintain the total density of tumour cells below its carrying capacity and prevents malignant tumour cells from taking over the whole population. We have also shown that slightly changing the immunotolerance rate along with the natural death rate of competent T cells improves the immune check-point inhibitor immunotherapy efficacy and that it can bring tumours from escape to eradication.

\subsection{Biological interpretations} Taking for granted the existence of a continuous malignancy trait in tumour cells, that we relate to a `degree of stemness' or de-differentiation potential, and similarly, of a continuous potential of tumour cell-kill in lymphocytes at the contact of tumour cells, we have qualitatively produced scenarios that reproduce the three $E$s of immunoediting. We have shown that the initial malignancy trait of tumour cells is affected by the immune response, with or without boosting by $ICI$ therapy, and that it will typically concentrate on a pointwise value, meaning that tumour cells as a population organise their stemness trait around a fixed dominant characteristic. Note that this concentration effect is an expected consequence of the choice of a Lotka-Volterra, here non local, model for the cancer cell population dynamics.

Whether this sharp malignancy trait is increased or decreased by the immune response cannot {\it a priori} be decided, as its determinants depend in a complex way on the entangled functions $d$, $r$, $\mu$ and $\varphi$ that govern the proliferation of the tumour cell population. If this model has some relevance with the reality of antitumoral immune response, it means that the effect of lymphocytes attacking a tumour may as well increase or decrease its stemness, which to the best of our knowledge is not inconsistent with biological observations so far. From a therapeutic point of view, we have shown, as proofs of concept, numerical case studies in which a tumour can be brought from escape to extinction, or at least equilibrium, by continuous delivery of $ICI$s.

It should also be mentioned that the design of this model, its characteristics and the results of its mathematical analysis can be enlightened by informal opinions heard in conversations or oral reports among oncologists, such as {\it ``This is a very aggressive tumour, very undifferentiated and escaping all drugs''} for tumour cell populations, or {\it ``Initially vigorous immune cells can, under the influence of immunotolerance induced by tumour cells, lose their vigor and become exhausted''} for immune cell populations. Such informal representation of a common aggressiveness potential in cell populations may be seen as tentatively formalised by the present phenotype-structured model of tumour-immune interactions.

\subsection{Possible generalisations}
Firstly, we plan to extend the model considered in this paper to carry out a mathematical study of tumour-response interactions, taking into account non-genetic instability, which may be considered as mediated by random epimutations in populations of tumour cells. In this respect, a modelling approach analogous to the one presented in \cite{RCLCJ}, would consist in modifying system \eqref{ID} as follows:
 \begin{equation}
 \label{pde}
 \left\{
 \begin{array}{ll}
 \displaystyle\frac{\partial n}{\partial t}(t,x)=\left[r(x) - d(x) \rho(t)-\mu(x)\varphi(t,x)\right]n(t,x) + \beta\frac{\partial^2 n}{\partial x^2}(t,x),& \\[0.3cm]
 \displaystyle\frac{\partial \ell}{\partial t}(t,y)=p(t,y)-\left(\nu(y)\rho(t)+k_1\right)\ell(t,y),&\\[0.3cm]
 \displaystyle\frac{\partial p}{\partial t}(t,y)=\chi(t,y)p(t,y)-k_2p^{2}(t,y),
 \end{array}%
 \right.
 \end{equation}
with Neumann boundary conditions at $x=0$ and $x=1$ for the cancer cell density $n(t, \cdot)$. The linear diffusion operator $\textstyle \beta\frac{\partial^2 n}{\partial^2 x}(t,x)$, with $0<\beta \ll 1$, represents here a malignancy phenotype lability (uncertainty) linked to the extreme plasticity of cancer cells \cite{Shen}, that are able to vary their phenotype in response to any (drug or other environmental) insult.

Another natural way to extend our work would be to introduce a population of antigen-presenting cells (APCs), that recognises a tumour antigen as their cognate one to activate na\"ive T-cells, instead of the time-independent shortcut function $\omega(x,y)$ (see Section \ref{sec3}). Delays might also naturally be introduced in this bidirectional communication process.

Future research perspectives, from the point of view of confronting the model to data, are to identify its parameters, making use of preclinical and clinical data on the growth of in-vivo tumours in laboratory rodents and in melanoma patients exposed to {\it ICI} therapies. This, however, will necessarily rely on long-term collaborations with teams of laboratory experimentalists and clinicians, towards whom we have here only set this physiologically based model as a basis for interactive discussions to assess it qualitatively.

Finally, as exemplified in \cite{P-al, Olivier2019}, it would be relevant to address the numerical optimal control of model~\eqref{ID} in order to identify possibly optimal delivery schedules for the {\it ICI} therapies, which will also be intended in the framework of an interactive collaboration with experimentalists and clinicians.

\section*{Acknowledgments}

The authors are gratefully indebted to Beno\^it Perthame and Luis Almeida for enriching discussions about model design, and to Alexandre Poulain for both model design and advice in the simulations in Matlab.

\end{document}